\documentclass[a4paper,12pt,reqno]{amsart}
\usepackage{amsfonts}
\usepackage{amsmath}
\usepackage{amssymb}
\usepackage[a4paper]{geometry}
\usepackage{mathrsfs}
\usepackage{csquotes}
\usepackage{enumitem}
\usepackage{dirtytalk}
\usepackage{comment}
\usepackage{nccmath}
\usepackage{pgfplots}
\pgfplotsset{compat=1.17}

\usepackage[colorlinks]{hyperref}
\usepackage[symbol]{footmisc}

\numberwithin{equation}{section}
\theoremstyle{plain}
\newtheorem{thm}{Theorem}[section]
\newtheorem{prop}[thm]{Proposition}
\newtheorem{cor}[thm]{Corollary}

\theoremstyle{definition}
\newtheorem{defn}[thm]{Definition}
\newtheorem{rem}[thm]{Remark}

\newtheorem*{conjecture}{Conjecture}

\newcommand\norm[1]{\left\lVert#1\right\rVert}

\def\e[#1]{{\textrm{e}}^{#1}}

\newcommand{\dist}{\operatorname{dist}}
\newcommand{\dd}{dx}
\allowdisplaybreaks[1]
\begin{document}

   \title[Sharp gradient stability of the $L^{p}$-Poincar\'e inequality]
 {Explicit sharp gradient stability of the $L^{p}$-Poincar\'e inequality and fundamental gaps for the Dirichlet $p$-Laplacian operator}

 \author[N. Yessirkegenov]{Nurgissa Yessirkegenov}
\address{
  Nurgissa Yessirkegenov:
  \endgraf
  KIMEP University, Almaty, Kazakhstan
     \endgraf
  {\it E-mail address} {\rm nurgissa.yessirkegenov@gmail.com}
  }

  \author[A. Zhangirbayev]{Amir Zhangirbayev}
\address{
  Amir Zhangirbayev:
 \endgraf
   SDU University, Kaskelen, Kazakhstan
  \endgraf
  and 
  \endgraf
  Institute of Mathematics and Mathematical Modeling, Kazakhstan
   \endgraf
  {\it E-mail address} {\rm
amir.zhangirbayev@gmail.com}
  }

\thanks{This research is funded by the Committee of Science of the Ministry of Science and Higher Education of the Republic of Kazakhstan (Grant No. AP32723635).}


     \begin{abstract} 
     In this paper, we establish sharp gradient stability of the $L^{p}$-Poincar\'e inequality with explicit constants on bounded convex domains for any $1<p<\infty$. These estimates control
     the distance to the optimizers in the strongest possible norm with the optimal stability exponent $\max\{2,p\}$, exhibiting the same phenomenon initially observed by Figalli and Zhang
     \cite{figalli2022sharp} for the $L^p$-Sobolev inequality. In the case $1<p<2$, we also establish a sharp upper bound gradient stability estimate with an explicit constant and the optimal exponent $p$, matching the exponent obtained by Ignat and Zuazua \cite{ignat2025optimal}. By extending the lower stability estimate to general bounded domains with $C^2$ boundary, we are able to confirm the conjecture posed by Fleckinger-Pell\'e and Tak\'a\v{c} \cite[Section 8]{fleckinger2002improved}. Finally, we apply the stability inequality to obtain positive fundamental gap estimates in terms of the inradius for the Dirichlet $p$-Laplacian.
     \end{abstract}
     \maketitle

\section{Introduction}

\subsection{$L^{p}$-Poincar\'e inequality and Dirichlet $p$-Laplacian}

Let $1 < p < \infty$ and $\Omega \subset \mathbb{R}^n$ be an open bounded connected set with $n\geq 2$. The classical Poincar\'e inequality (see \cite{GilbargTrudinger1977} or \cite[Section 5.8.1]{evans2022partial}) states that there exists a constant $c>0$, depending only on $\Omega$ and $p$, such that
\begin{align}\label{Poincare}
\int_{\Omega}|\nabla u|^{p}dx \geq c\int_{\Omega}|u|^{p}dx
\end{align}
for all $u\in W_{0}^{1,p}(\Omega)$. 

The inequality (\ref{Poincare}) and its various generalizations serve as fundamental tools in the analysis of nonlinear partial differential equations (PDEs) and spectral theory. For example, in \cite{brezis1997blow}, Brezis and V\'azquez obtained an improved $L^{2}$-Poincar\'e inequality, which they used together with a semistability criterion to study singular extremal solutions of
\begin{align}\label{BV97 pde}
\begin{cases}
-\Delta u = \lambda f(u) & \text{in } \Omega, \\
u = 0 & \text{on } \partial\Omega,
\end{cases}
\end{align}
where $\Omega$ is a bounded smooth domain, $\lambda > 0$ and $f$ is
continuous, positive, increasing and convex on $[0,\infty)$ with
$f(s)/s \to \infty$ as $s \to \infty$.  Another application to PDEs appears in the work of Fleckinger-Pell\'e and Tak\'a\v{c} \cite{fleckinger2002improved},  where they also established an improved version of the inequality (\ref{Poincare}) and then applied it to study the existence of a weak solution to the degenerate quasilinear boundary value problem at resonance:
\begin{align}\label{takac pde}
  \begin{cases}
    -\Delta_p u = \lambda_1 |u|^{p-2} u + f(x) & \text{in } \Omega, \\
    u = 0 & \text{on } \partial\Omega.
  \end{cases}
\end{align}
These models arise naturally in the study of thermal explosions and spontaneous ignition \cite{gelfand1963some}, as well as in the modeling of compressible fluids within homogeneous, isotropic, rigid porous media (see \cite{drabek2007plaplacian}).

The best constant and the set of optimizers in the inequality (\ref{Poincare}) are directly related to the spectral information of the Dirichlet $p$-Laplacian operator:
\begin{align}\label{eq first eigenvalue}
\lambda_{1}(\Omega,p):=\min_{u\in W_{0}^{1,p}(\Omega)\setminus\{0\}}
\frac{\int_{\Omega}|\nabla u|^{p}dx}{\int_{\Omega}|u|^{p}dx}.
\end{align}
That is, the first eigenvalue $\lambda_{1}(\Omega,p)$ of $-\Delta_{p}$ is the best constant in the inequality (\ref{Poincare}) and the set of optimizers is $\mathbb{R}u_{1}$, where $u_{1}$ is a corresponding positive first eigenfunction. Equivalently, introducing
\begin{align*}
\mathcal{S}_{p}:=\left\{u\in W^{1,p}_{0}(\Omega):\norm{u}_{L^{p}(\Omega)}=1\right\},
\end{align*}
we can also write
\begin{align*}
\lambda_{1}(\Omega,p)=\min_{u\in\mathcal{S}_{p}}\int_{\Omega}|\nabla u|^{p}dx.
\end{align*}
The eigenvalue problem itself reads as follows:
\begin{align*}
\begin{cases} 
-\Delta_{p} u = \lambda|u|^{p-2} u & \text{in } \Omega, \\ 
u = 0 & \text{on } \partial \Omega,
\end{cases}
\end{align*}
where $\Delta_{p}(\cdot):=\operatorname{div}\left(|\nabla(\cdot)|^{p-2} \nabla(\cdot)\right)$ is the $p$-Laplacian operator. We say that $\lambda \in \mathbb{R}$ is an eigenvalue of $-\Delta_p$ and $u \in W_0^{1,p}(\Omega) \setminus \{0\}$ is an associated eigenfunction if
\begin{align}\label{eq weak eigenfunction}
\int_{\Omega} |\nabla u|^{p-2} \langle\nabla u,\nabla\varphi\rangle dx = \lambda \int_{\Omega} |u|^{p-2} u \varphi dx \quad \text{for any } \varphi \in W_0^{1,p}(\Omega).
\end{align}
The pair $(\lambda,u)$ is then called an eigenpair of the Dirichlet $p$-Laplacian.

When $p=2$, it is well-known that the whole spectrum of $-\Delta$ is discrete and consists of an unbounded sequence of positive real eigenvalues. In fact, one can use the Courant-Fischer minimax formula 
\begin{align*}
\lambda_k(\Omega,2):=\min_{\substack{X_k\subset W_{0}^{1,2}(\Omega)\\\dim X_k=k}}
\max_{u\in X_k\cap\mathcal{S}_{2}}\int_{\Omega}|\nabla u|^{2}dx,
\quad k\in\mathbb{N},
\end{align*}
to characterize the whole spectrum. However, when $p\neq 2$, the nonlinearity of the operator $-\Delta_{p}$ and the lack of the Hilbert space structure in $W_{0}^{1,p}(\Omega)$ make the analysis of the spectrum particularly difficult. The Ljusternik-Schnirelman theory \cite{anane1996second, struwe2008variational} is used instead to construct the variational sequence of eigenvalues. Specifically, one defines the Krasnoselskii genus, denoted by $\gamma(\mathcal{A})$, for any nonempty symmetric set $\mathcal{A}\subset W_{0}^{1,p}(\Omega)\setminus\{0\}$ as follows
\begin{align*}
\gamma(\mathcal{A}) = \inf\{m \in \mathbb{N} : \exists \text{ a continuous odd map } f : \mathcal{A} \to \mathbb{R}^m \setminus \{0\}\}
\end{align*}
with $\gamma(\mathcal{A}):=+\infty$ if no such $f$ exists. Let $\Sigma_{k}$ denote the family of all nonempty compact symmetric subsets $\mathcal{A}\subset\mathcal{S}_{p}$ with $\gamma(\mathcal{A})\geq k$, where compactness is understood in $W_{0}^{1,p}(\Omega)$. Then the sequence of Ljusternik-Schnirelman eigenvalues is defined by
\begin{align}\label{eq LS eigenvalues}
\lambda_k(\Omega,p):=\inf_{\mathcal{A}\in\Sigma_k}\max_{u\in\mathcal{A}}\int_{\Omega}|\nabla u|^{p}dx.
\end{align}
Although \eqref{eq LS eigenvalues} for $p=2$ recovers the whole spectrum of the standard Laplacian \cite{cuesta2000fucik}, the variational sequence does not exhaust the spectrum in general. The existence of non-Ljusternik-Schnirelman eigenvalues was a major open problem and was recently resolved by Bobkov for suitable planar rectangles and $p>2$ sufficiently close to two \cite{bobkov2026nonljusternik}. We refer to \cite{browder1970nonlinear, szulkin1988ljusternik, drabek1999resonance, fucik2006spectral} for several seminal works on this topic.

\subsection{The fundamental gap I}

While we still do not know the full description of the spectrum of the operator $-\Delta_{p}$, we do know that no eigenvalue exists between the first $\lambda_1(\Omega, p)$ and the second $\lambda_2(\Omega, p)$. The first eigenvalue $\lambda_{1}(\Omega,p)$ is simple, strictly positive and isolated (see \cite{anane1987simplicite,lindqvist1990equation}). Since there are no eigenvalues between these first two, this naturally leads to a question of geometric and physical significance: exactly how large is the gap between them? This difference is usually referred to as the fundamental gap and denoted by
\begin{align*}
\Gamma(\Omega, p) := \lambda_2(\Omega, p) - \lambda_1(\Omega, p) > 0.
\end{align*}
We use a mountain-pass characterization of the second eigenvalue for an open bounded connected domain $\Omega$ and $1<p<\infty$ by Cuesta, de Figueiredo and Gossez \cite{cuesta1999beginning}:
\begin{align}\label{eigenvalue second}
\lambda_{2}(\Omega,p)
=\inf_{\gamma\in\Psi(u_{1},-u_{1})}
\max_{t\in[0,1]}\int_{\Omega}|\nabla\gamma(t)|^{p}dx,
\end{align}
where 
\begin{align*}
\Psi(u_{1},-u_{1}):=\left\{\gamma\in C([0,1],\mathcal{S}_{p}):\gamma(0)=u_{1},\ \gamma(1)=-u_{1}\right\}
\end{align*}
and $\norm{u_{1}}_{L^{p}(\Omega)}=1$. We also refer to \cite{brasco2013hong} for an extension of (\ref{eigenvalue second}) to arbitrary open sets of finite measure.

In the linear case $p=2$, the problem of estimating the fundamental gap has a rich history. We do not attempt to survey the extensive literature and, instead, provide only some key developments. In 1983, van den Berg \cite{vandenberg1983condensation} observed that the bound
\begin{align*}
\Gamma(\Omega,2) \geq \frac{3\pi^{2}}{\mathrm{diam}(\Omega)^{2}}
\end{align*}
holds for many convex domains $\Omega$ and suggested that the bound holds for all such domains. Later, Yau \cite{yau1987nonlinear} and Ashbaugh with Benguria \cite{ashbaugh1989optimal}
independently conjectured that the constant $\frac{3\pi^{2}}{\mathrm{diam}(\Omega)^{2}}$ is actually sharp not only for the standard Dirichlet Laplacian $-\Delta$, but also for the Schr\"odinger operator $-\Delta+V$ in all dimensions $n\geq 1$, where $V$ is a convex potential on $\Omega\subset \mathbb{R}^{n}$. While Lavine \cite{lavine1994eigenvalue} proved this conjecture in one dimension, the higher-dimensional case was much harder to show. A major step toward its resolution was taken by Singer, Wong, Yau and Yau \cite{singer1985estimate}, who established the lower bound $\frac{\pi^{2}}{4\mathrm{diam}(\Omega)^{2}}$. This was later improved to $\frac{\pi^{2}}{\mathrm{diam}(\Omega)^{2}}$ by Yu and Zhong \cite{yu1986lower}. The conjecture was ultimately resolved by Andrews and Clutterbuck \cite{andrews2011proof} who demonstrated that the sharp lower bound is indeed $\frac{3\pi^{2}}{\mathrm{diam}(\Omega)^{2}}$. Their approach introduced several novel concepts such as a modulus of continuity to the heat equation and a modulus of concavity for $\log u_{1}$ (see \cite{brascamp1976extensions, andrews2009lipschitz, andrews2009time}). Shortly after, Ni \cite{ni2013estimates} provided an alternative proof by an elliptic argument and the improved log-concavity from \cite{andrews2011proof}. We refer interested readers to \cite{ashbaugh2006sharp, seto2019sharp} for comprehensive surveys on the subject and to \cite{oden1999spectral, he2020fundamental, dai2021fundamental, andrews2012eigenvalue, bourni2022vanishing, amato2025geometric, clutterbuck2025constant} for other crucial developments, including studies of the fundamental gap on spheres as well as in curved and hyperbolic spaces.

On the other hand, the nonlinear case $p\neq 2$ is much less understood and results in this setting are still under development. In \cite{sakaguchi1987concavity}, Sakaguchi showed that the positive first eigenfunction is log-concave on bounded convex domains in any dimension $n\geq 2$. But many of the other crucial tools used in the linear case, such as the modulus of concavity, do not extend directly when $p\neq2$ and $n\geq2$ (see \cite{chen2026two}). Because these techniques fail in higher dimensions, earlier work concentrated largely on one-dimensional settings. Cheng, Lian and Wang studied gap minimization for single-well potentials
\cite{cheng2014dual}. Andrews, Clutterbuck and Hauer obtained further interval comparisons under Robin boundary conditions,
including results for the $p$-Laplacian \cite{andrews2021fundamental}. We refer to \cite{chen2020remarks, wen2026maximal} for related results regarding the $p$-Laplacian in the one-dimensional case.

Before considering the $n\geq 2$ case, let us introduce the following notation
\begin{align*}
\mathcal{C}_{n}:=\left\{\Omega\subset\mathbb{R}^{n}:\Omega\text{ is an open bounded convex domain}\right\}.
\end{align*}
A recent contribution in higher dimensions is from \cite{yessirkegenov2026stability}.
There, the authors of this paper obtained the following Dirichlet
$p$-Laplacian fundamental gap estimate for $n\geq2$, $p\geq2$ and $\Omega\in\mathcal{C}_{n}$:
\begin{align}\label{yz26 intro}
  \Gamma(\Omega,p)
  \geq
  c_{\mathrm{deg}}(p)
  \left(\frac{\pi_p}{\mathrm{diam}(\Omega)}\right)^p
  \inf_{c\in\mathbb R}\int_{\Omega}|u_{2}-cu_{1}|^{p}dx.
\end{align}
Here, the eigenfunctions are normalized and the constants are
\begin{align}
c_{\mathrm{deg}}(p)
&:=\begin{cases}
1, & p=2,\\
(p-1)(r_0+1)^{2-p}, & p>2,
\end{cases}\label{eq degenerate constant}\\
\pi_p
&=2\int_0^\infty\frac{ds}{1+s^p/(p-1)}
=\frac{2\pi(p-1)^{1/p}}{p\sin(\pi/p)},\notag
\end{align}
where $r_0$ is a positive solution to the equation $r^{p-1}-(p-1) r-(p-2)=0$. Although \eqref{yz26 intro} still contains an eigenfunction-dependent
factor, it provided the first quantitative estimate in higher dimensions $n\geq 2$. In addition, the result actually came from the problem of stability. In short, one can see why from the spectral decomposition of the Dirichlet Laplacian in $L^{2}(\Omega)$:
\begin{align}\label{spectral decom}
\int_\Omega|\nabla u|^2dx-\lambda_1(\Omega,2)\int_\Omega|u|^2dx\geq(\lambda_2(\Omega,2)-\lambda_1(\Omega,2))
\inf_{c\in\mathbb R}\int_\Omega|u-cu_1|^2dx
\end{align}
for every $u\in W_0^{1,2}(\Omega)$.
The constant on the right-hand side of (\ref{spectral decom}) is sometimes called the \textit{stability constant}.

\subsection{Stability}

Historically, the problem of deriving stability estimates of inequalities was first posed by Brezis and Lieb \cite{brezis1985sobolev} for the $L^{2}$-Sobolev inequality. Then, Bianchi and Egnell resolved the question in \cite{bianchi1991note}. Since then, the study of stability estimates for functional and geometric inequalities has attracted considerable attention with landmark contributions including \cite{fusco2008sharp, carlen2013stability, brasco2015faber, carlen2017duality, figalli2022sharp, figalli2023sharp, Koenig2023, Koenig2025, dolbeault2025sharp}. We also refer to \cite{figalli2011shape, carlen2013stability, dolbeault2013improved, cicalese2013droplet, figalli2013isoperimetric, figalli2022strong} for various applications of stability.

A major stability result closely related to our work is due
to Figalli and Zhang \cite{figalli2022sharp}, who established
sharp gradient stability for the $L^p$-Sobolev inequality
for every $1<p<n$.
Let $S$ denote the optimal constant in the Sobolev inequality
$\|\nabla u\|_{L^p(\mathbb{R}^n)}
\geq S\|u\|_{L^{p^*}(\mathbb{R}^n)}$
and let $\mathcal{M}$ be the family of all nonzero
Sobolev optimizers.
Then, for every nonzero $u\in\dot W^{1,p}(\mathbb{R}^n)$,
we have
\begin{align}
\frac{\norm{\nabla u}_{L^{p}(\mathbb{R}^{n})}}{\norm{u}_{L^{p^{*}}(\mathbb{R}^{n})}}-S\geq c\inf_{v\in\mathcal{M}}
\left(\frac{\norm{\nabla(u-v)}_{L^{p}(\mathbb{R}^{n})}}{\norm{\nabla u}_{L^{p}(\mathbb{R}^{n})}}\right)^{\max\{2,p\}},
\label{eq figalli zhang stability}
\end{align}
where $p^{*}:=\frac{np}{n-p}$ and $c>0$ depends only on $n$ and $p$.
The exponent $\max\{2,p\}$ in (\ref{eq figalli zhang stability}) is not only optimal, but also seems to be a recurring feature of quantitative stability in $L^{p}$-based spaces. For instance, this phenomenon was observed to hold also for related inequalities such as Hardy-Sobolev
\cite{deng2025stability}, Hardy-Sobolev-Maz'ya \cite{dai2025sharp} and
Caffarelli-Kohn-Nirenberg
\cite{deng2025caffarelli,deng2024gradient,zhou2026stability}. Furthermore, the same exponent appears in substantially different analytic and geometric settings such as affine Sobolev inequalities \cite{fan2026sharp,fan2026sharp2}, spinorial Sobolev
inequalities \cite{wang2026stability}, harmonic extension inequalities
\cite{frank2025sharp} and curvature inequalities \cite{frank2024sharp, peteranderl2026sharp}.

A local upper bound for the Sobolev deficit was obtained by Ignat and Zuazua \cite[Corollary~3.5 and Remark~3.6]{ignat2025optimal}: there exist constants $C>0$ and $\varepsilon_{0}>0$, depending only on $n$ and $p$, such that, for every nonzero $u\in\dot W^{1,p}(\mathbb{R}^{n})$ satisfying
\begin{align*}
\inf_{v\in\mathcal{M}}
\frac{\norm{\nabla(u-v)}_{L^{p}(\mathbb{R}^{n})}}{\norm{\nabla u}_{L^{p}(\mathbb{R}^{n})}}
<\varepsilon_{0},
\end{align*}
we have
\begin{align}\label{eq ignat zuazua upper bound}
\frac{\norm{\nabla u}_{L^{p}(\mathbb{R}^{n})}}{\norm{u}_{L^{p^{*}}(\mathbb{R}^{n})}}-S
\leq C\inf_{v\in\mathcal{M}}
\left(\frac{\norm{\nabla(u-v)}_{L^{p}(\mathbb{R}^{n})}}{\norm{\nabla u}_{L^{p}(\mathbb{R}^{n})}}\right)^{\min\{2,p\}}.
\end{align}
The exponent $\min\{2,p\}$ is optimal. In particular, in the singular range $1<p<2$, this upper bound has exponent $p$, whereas the lower bound in (\ref{eq figalli zhang stability}) has exponent $2$.

Once stability is established, the next question is whether its constant can be made explicit. Non-constructive methods, such as concentration-compactness, generally prove the existence of a positive constant without providing a computable value. This has motivated several constructive approaches \cite{dolbeault2021functional, bonforte2023constructive, chen2024stability,brigati2024logarithmic, brigati2024stability,chen2025optimal,dolbeault2025sharp,chen2026optimal,chen2025asymptotically, bonforte2025stability}. In \cite{yessirkegenov2026stability}, the authors build on the exact-remainder and weighted Poincar\'e methods used in \cite{do2023lp,cazacu2024caffarelli,do2024scale,shaimerdenov2025sharp,do2026sharp} to obtain the following stability estimate with an explicit constant: let $p,n\geq 2$ and $\Omega\in\mathcal{C}_{n}$. Then, for every $u\in C_{0}^{\infty}(\Omega)$, we have
\begin{align}\label{eq yz26 stability}
\int_{\Omega}|\nabla u|^{p}dx-\lambda_{1}(\Omega,p)\int_{\Omega}|u|^{p}dx\geq c_{\mathrm{deg}}(p)\left(\frac{\pi_{p}}{\mathrm{diam}(\Omega)}\right)^{p}
\inf_{c\in\mathbb{R}}\int_{\Omega}|u-cu_{1}|^{p}dx.
\end{align}
Extending (\ref{eq yz26 stability}) by density and applying it to an $L^{p}$-normalized second eigenfunction, they obtained the fundamental gap estimate (\ref{yz26 intro}). 

Earlier, Fleckinger-Pell\'e and Tak\'a\v{c}
\cite{fleckinger2002improved} established the following stability estimate
for the $L^p$-Poincar\'e inequality \eqref{Poincare}:
let $p>2$ and let $\Omega\subset\mathbb R^n$ be a bounded connected
domain with connected $C^2$ boundary. Then there exists a constant
$c>0$, depending on $\Omega$ and $p$, such that, for every $u\in W_0^{1,p}(\Omega)$,
\begin{multline}\label{eq FT degenerate}
\int_\Omega \left|\nabla u\right|^{p}dx-\lambda_1\int_\Omega \left|u\right|^{p}dx
\geq c\left(
|u^\parallel|^{p-2}
\int_\Omega|\nabla u_1|^{p-2}|\nabla u^\perp|^2\dd
\right.\\\left.
+\int_\Omega \left|\nabla u^\perp\right|^{p}dx\right),
\end{multline}
where the decomposition is given by \begin{align}\label{eq projection}
u^\parallel=\frac{\int_\Omega uu_1dx}{\int_\Omega u_1^2dx},
\qquad u^\perp=u-u^\parallel u_1.
\end{align}
In the singular case $1<p<2$, Fleckinger-Pell\'e and Tak\'a\v{c}
\cite[Section~8, p.~969]{fleckinger2002improved} posed the following conjecture.
\begin{conjecture}
Let $1<p<2$ and $\Omega\subset\mathbb{R}^n$ be a bounded domain with $C^{2}$ boundary. Then, there exists a constant $c>0$, depending on $\Omega$ and $p$, such that for all $u \in W_0^{1,p}(\Omega)$, we have
\begin{align}\label{eq:8.1}
\int_\Omega |\nabla u|^p \,\mathrm{d}x - \lambda_1 \int_\Omega |u|^p \,\mathrm{d}x \ge c \left(\max\left\{|u^{\parallel}|,\norm{\nabla u^{\perp}}_{L^{p}(\Omega)}\right\}\right)^{-(2-p)}
\norm{\nabla u^{\perp}}_{L^{p}(\Omega)}^{2}.
\end{align}
Equivalently, with some other constant $c' > 0$, we have
\begin{align*}
\int_\Omega |\nabla u|^p \,\mathrm{d}x - \lambda_1 \int_\Omega |u|^p \,\mathrm{d}x
\ge c' \norm{\nabla u}_{L^{p}(\Omega)}^{-(2-p)}
\norm{\nabla u^{\perp}}_{L^{p}(\Omega)}^{2}.
\end{align*}
\end{conjecture}
Related analysis of the singular energy near the set of optimizers and of the Palais-Smale condition was developed by Dr\'abek and Tak\'a\v{c} \cite{drabek2007poincare}. We also note the work of Bobkov and Kolonitskii \cite{bobkov2023improved} where they provided qualitative improvements of the nonlinear Friedrichs inequality in the form of (\ref{eq FT degenerate}). 

\subsection{The fundamental gap II}

Most recently, Chen and Hauer \cite{chen2026fundamental} extended the results of \cite{yessirkegenov2026stability} to convex potentials and removed the eigenfunction-dependent factor in (\ref{yz26 intro}): let $n\geq2$, $p\geq2$ and $\Omega\in\mathcal{C}_{n}$. Then, we have
\begin{align}\label{eq diameter gap}
\Gamma(\Omega,p)
  \geq  c_{\mathrm{deg}}(p)
  \left(\frac{\pi_p}{\mathrm{diam}(\Omega)}\right)^p.   
\end{align}
Here, we state the zero-potential case for convenience. In the case of $n=1$ and $1<p<\infty$, they resolved the question of the sharp fundamental gap completely \cite[Theorem~1.1]{chen2026fundamental}. In dimensions $n\geq 2$, they showed that
\begin{align*}
\inf_{\Omega\in\mathcal{C}_{n}}
\mathrm{diam}(\Omega)^p\Gamma(\Omega,p)
=
\begin{cases}
0, & 1<p<2,\\
3\pi^2, & p=2.
\end{cases}
\end{align*}
When $p>2$, they proved that the infimum is positive and attained
by some $\Omega\in\mathcal{C}_{n}$. More precisely, their lower bound for
the gap in terms of the first eigenvalue shows that, at fixed
diameter, the gap diverges as the inradius tends to zero. Therefore, the minimizing sequences cannot collapse. Although the minimizing
domain and the exact value of the infimum remain unknown, the optimal diameter-normalized gap is
comparable to $n^{p-2}$. Unlike
the case $p=2$, the optimal constant depends on the
dimension \cite[Theorem 1.5]{chen2026fundamental}.

\subsection{Main contributions}

In this paper, we prove the gradient stability estimate for the $L^{p}$-Poincar\'e inequality with explicit constants over the full range $1<p<\infty$: let $n\geq2$, $\Omega\in\mathcal{C}_{n}$ and $\rho>p$ be the unique solution of
\begin{align*}
\rho^{p-1}(\rho-p)-\lambda_{1}(\Omega,p)
\left(\frac{\mathrm{diam}(\Omega)}{\pi_{p}}\right)^{p}=0.
\end{align*}
Then, for every $0\neq u\in W_{0}^{1,p}(\Omega)$, we have
\begin{align}\label{eq gradient stability intro}
1-\lambda_{1}(\Omega,p)\frac{\norm{u}_{L^{p}(\Omega)}^{p}}{\norm{\nabla u}_{L^{p}(\Omega)}^{p}}
\geq C_{\Omega,p}
\inf_{c\in\mathbb{R}}\left(\frac{\norm{\nabla(u-cu_{1})}_{L^{p}(\Omega)}}{\norm{\nabla u}_{L^{p}(\Omega)}}\right)^{\alpha},
\end{align}
where $\alpha:=\max\{2,p\}$ and
\begin{align}\label{eq stability constant intro}
C_{\Omega,p}:=\frac{1}{(1+\rho)^{\max\{2,p\}}}\times
\begin{cases}
c_{\mathrm{deg}}(p),&p\geq2,\\[1ex]
\underline{c}_{\mathrm{sing}}(p)\left(\dfrac{p-1}{2p-1}\right)^{2-p},&1<p<2.
\end{cases}
\end{align}
Moreover, $\alpha$ is optimal.
For $1<p<2$, the constant $\underline{c}_{\mathrm{sing}}(p)$ is defined by
\begin{align*}
\underline{c}_{\mathrm{sing}}(p)
&:=\inf_{s^{2}+t^{2}>0}
\frac{\left(t^{2}+s^{2}+2s+1\right)^{p/2}-1-ps}
{\left(\sqrt{t^{2}+s^{2}+2s+1}+1\right)^{p-2}(s^{2}+t^{2})}
\in\left(0,\frac{p(p-1)}{2^{p-1}}\right].
\end{align*}
We also prove that
\begin{align}\label{eq rho upper bound intro}
\rho<p+\lambda_{1}(\Omega,p)^{1/p}\frac{\mathrm{diam}(\Omega)}{\pi_{p}}.
\end{align}
Substituting (\ref{eq rho upper bound intro}) into (\ref{eq stability constant intro}) gives the explicit lower bounds for $C_{\Omega,p}$ stated in Remark \ref{rem explicit stability}.

The estimate (\ref{eq gradient stability intro}) controls the distance to the optimizers in the full gradient norm of $W_{0}^{1,p}(\Omega)$. However, in our proof, the stronger norm comes at the cost of a weaker stability constant. Its stability exponent is $2$ for $1<p\leq2$ and $p$ for $p>2$, exhibiting the same transition as in the sharp gradient stability of the Sobolev inequality proved by Figalli and Zhang \cite{figalli2022sharp}. In the singular case, (\ref{eq gradient stability intro}) confirms the conjecture of Fleckinger-Pell\'e and Tak\'a\v{c} (\ref{eq:8.1}) for $\Omega\in \mathcal{C}_n$ (see Theorem~\ref{thm takac conjecture general} for general bounded domains with $C^{2}$ boundary). To our knowledge, this is the first gradient stability estimate for an $L^{p}$-functional inequality that combines the optimal stability exponent with an explicit stability constant throughout the full range $1<p<\infty$. 

In the singular range $1<p<2$, we also prove a sharp upper bound gradient stability: let $\Omega\in\mathcal{C}_{n}$, $n\geq 2$ and the normalized Poincar\'e deficit be defined by
\begin{align*}
\delta_{\Omega,p}(u):=1-\lambda_{1}(\Omega,p)
\frac{\norm{u}_{L^{p}(\Omega)}^{p}}{\norm{\nabla u}_{L^{p}(\Omega)}^{p}}.
\end{align*}
Then, for every $0\neq u\in W_{0}^{1,p}(\Omega)$, we have
\begin{multline}\label{eq two sided gradient stability intro}
\frac{\underline{c}_{\mathrm{sing}}(p)}{(1+\rho)^{\max\{2,p\}}}\left(\dfrac{p-1}{2p-1}\right)^{2-p}\inf_{c\in\mathbb{R}}
\left(\frac{\norm{\nabla(u-cu_{1})}_{L^{p}(\Omega)}}{\norm{\nabla u}_{L^{p}(\Omega)}}\right)^{2}
\leq\delta_{\Omega,p}(u)\\
\leq\overline{c}_{\mathrm{sing}}(p)
\left(\frac{2p-1}{p-1}\right)^{p}
\inf_{c\in\mathbb{R}}
\left(\frac{\norm{\nabla(u-cu_{1})}_{L^{p}(\Omega)}}{\norm{\nabla u}_{L^{p}(\Omega)}}\right)^{p},
\end{multline}
where $\overline{c}_{\mathrm{sing}}(p)$ is defined in (\ref{eq singular upper constant}). Both the lower exponent $2$ and the upper exponent $p$ are optimal.

As one can see, the same optimal exponent as in the Sobolev upper bound of Ignat and Zuazua (\ref{eq ignat zuazua upper bound}) occurs for the Poincar\'e inequality as well. Additionally, \eqref{eq two sided gradient stability intro} holds globally for every nonzero $u\in W_{0}^{1,p}(\Omega)$ without a smallness assumption on its distance to the optimizers. Meaning, it applies to every term of a minimizing sequence,
including any initial terms far from the optimizer set. The lower constant has the explicit geometric dependence given above, while the upper constant depends only on $p$. Ignat and Zuazua state their upper estimate with an unspecified constant $C$ and a sufficiently small threshold for the normalized gradient distance. We also note that convexity is needed for our lower bound, whereas our upper bound holds on arbitrary open bounded domains.

We then apply (\ref{eq gradient stability intro}) to the fundamental gap. Deriving a gradient version of the path-intersection argument in \cite[Lemma~6.3]{chen2026fundamental} and following the mountain-pass argument in \cite[proof of Theorem~1.7~(i)]{chen2026fundamental}, we get
\begin{align}\label{eq fundamental gap intro}
\Gamma(\Omega,p)\geq\frac{C_{\Omega,p}}{1-C_{\Omega,p}}\lambda_{1}(\Omega,p).
\end{align}
Equivalently, (\ref{eq fundamental gap intro}) reads $\Gamma(\Omega,p)\geq C_{\Omega,p}\lambda_{2}(\Omega,p)$. Thus the gradient stability estimate yields a fundamental gap bound in terms of the first eigenvalue for every $1<p<\infty$.

Finally, employing both sides of the inradius bounds of Brasco \cite[Theorem~1.2 and Remark~1.4]{brasco2018principal}, we get the following geometric estimate
\begin{align}\label{eq inradius gap intro}
\Gamma(\Omega,p)\geq K_{n,p}
\frac{r_{\Omega}^{\max\{2,p\}-p}}{\mathrm{diam}(\Omega)^{\max\{2,p\}}},
\end{align}
where $r_{\Omega}:=\sup_{x\in\Omega}\dist(x,\partial\Omega)$ is the inradius and
\begin{align*}
K_{n,p}:={}\frac{\pi_{p}^{p}}{
2^{p}\left(\frac{p+1}{2}+\frac{\lambda_{1}(B_{1}(0),p)^{1/p}}{\pi_{p}}\right)^{\max\{2,p\}}}\times\begin{cases}
c_{\mathrm{deg}}(p),&p\geq2,\\[1ex]
\underline{c}_{\mathrm{sing}}(p)\left(\dfrac{p-1}{2p-1}\right)^{2-p},&1<p<2.
\end{cases}
\end{align*}
Here $B_{1}(0)$ denotes the unit ball in $\mathbb{R}^{n}$.

For $p\geq2$, (\ref{eq inradius gap intro}) is a diameter bound with a smaller constant than the estimate of Chen and Hauer in (\ref{eq diameter gap}). For $1<p<2$, Chen and Hauer \cite[Theorems~1.2 and~1.3]{chen2026fundamental} constructed smooth domains in $\mathcal{C}_{n}$ collapsing to a segment whose fundamental gaps tend to zero. Hence no positive lower bound depending only on the diameter can hold in this range. The singular case of (\ref{eq inradius gap intro}) complements their result by providing an explicit positive bound involving both the inradius and the diameter. The factor $r_{\Omega}^{2-p}$ allows the lower bound to tend to zero as the inradius vanishes at fixed diameter.

\subsection{Strategy of the proof}

The proof relies on a simple observation: setting
$f:=u/u_1-c$, we have
\begin{align}\label{eq strategy gradient decomposition}
u-cu_1=u_1f,
\qquad
\nabla(u-cu_1)=u_1\nabla f+f\nabla u_1.
\end{align}
For $p\geq2$ (Theorem \ref{thm yz26}), we have
\begin{align}\label{eq strategy degenerate remainder}
\int_\Omega|\nabla u|^pdx
-\lambda_1(\Omega,p)\int_\Omega|u|^pdx
\geq c_{\mathrm{deg}}(p)
\int_\Omega|u_1\nabla f|^pdx.
\end{align}
For $1<p<2$ (Theorem \ref{thm poinc 1p2}), we have
\begin{align}
&\int_\Omega|\nabla u|^pdx
-\lambda_1(\Omega,p)\int_\Omega|u|^pdx\notag\\
&\qquad\geq\underline{c}_{\mathrm{sing}}(p)
\int_\Omega
\left(|\nabla u|
+\left|\frac{u}{u_1}\nabla u_1\right|\right)^{p-2}
|u_1\nabla f|^2dx.
\label{eq strategy singular remainder}
\end{align}
Both inequalities follow from \cite{apseit2025sharp} by the usual argument of integration by parts and the
first eigenfunction equation with bounds for the convexity remainder $C_p$.

As one can see, these inequalities already control the first term
$u_1\nabla f$ in (\ref{eq strategy gradient decomposition}).
We would now like to control the second term $f\nabla u_1$
as well. More precisely, we look for
an estimate of the form
\begin{align}\label{eq strategy ground state gradient bound}
\norm{f\nabla u_1}_{L^p(\Omega)}
\leq\rho\norm{u_1\nabla f}_{L^p(\Omega)}
\end{align}
with a constant $\rho>0$ independent of $u$.

The first eigenfunction equation provides a way to compare
these terms. Indeed, testing the weak formulation
(\ref{eq weak eigenfunction}) for $\varphi=u_1|f|^p$
and applying H\"older's inequality gives
\begin{align}
\norm{f\nabla u_1}_{L^p(\Omega)}^p
&\leq\lambda_1(\Omega,p)
\norm{u_1f}_{L^p(\Omega)}^p\notag\\
&\quad+p\norm{f\nabla u_1}_{L^p(\Omega)}^{p-1}
\norm{u_1\nabla f}_{L^p(\Omega)}.
\label{eq strategy testing bound}
\end{align}
The right-hand side of (\ref{eq strategy testing bound})
still contains $\norm{u_1f}_{L^p(\Omega)}$.
However, thanks to the weighted
Poincar\'e inequality, this term can be
controlled by $\norm{u_1\nabla f}_{L^p(\Omega)}$, which
is exactly what we need to get
(\ref{eq strategy ground state gradient bound}).
Here, $c$ is chosen to minimize
\begin{align*}
\int_\Omega|u-cu_1|^pdx,
\end{align*}
ensuring the weighted mean condition. Therefore, in (\ref{eq strategy testing bound}), we have
\begin{align*}
\norm{f\nabla u_1}_{L^p(\Omega)}^p
&\leq\lambda_1(\Omega,p)
\left(\frac{\mathrm{diam}(\Omega)}{\pi_p}\right)^p
\norm{u_1\nabla f}_{L^p(\Omega)}^p\\
&\quad+p\norm{f\nabla u_1}_{L^p(\Omega)}^{p-1}
\norm{u_1\nabla f}_{L^p(\Omega)}.
\end{align*}
This inequality now involves only
$\norm{u_1\nabla f}_{L^p(\Omega)}$ and
$\norm{f\nabla u_1}_{L^p(\Omega)}$. If the first quantity vanishes, so does the second.
Otherwise, dividing by
$\norm{u_1\nabla f}_{L^p(\Omega)}^p$
shows that their ratio is bounded by a constant depending
only on $\Omega$ and $p$, which means that (\ref{eq strategy ground state gradient bound}) does hold for some
$\rho>0$ independent of $u$. Applying the triangle inequality in 
(\ref{eq strategy gradient decomposition}) and (\ref{eq strategy ground state gradient bound}) finally gives
\begin{align}\label{eq strategy gradient conversion}
\norm{\nabla(u-cu_1)}_{L^p(\Omega)}
\leq(1+\rho)\norm{u_1\nabla f}_{L^p(\Omega)}.
\end{align}
The estimate (\ref{eq strategy gradient conversion}) then transfers the control provided by the remainder inequalities, (\ref{eq strategy degenerate remainder}) and
(\ref{eq strategy singular remainder}),
to the full gradient distance to the optimizers.

\subsection{Structure of the paper}

The article is organized as follows. In Section
\ref{sec preliminaries}, we introduce the $C_p$-functional and the remainder estimates for the Poincar\'e inequality in the degenerate and singular ranges. We also recall the weighted Poincar\'e inequality for log-concave weights and the bounds for the first Dirichlet $p$-Laplacian eigenvalue in terms of the inradius. Section \ref{sec main results} contains our gradient stability estimates and explicit bounds for the constants. We then state the resolution of the conjecture of Fleckinger-Pell\'e and Tak\'a\v{c}, including the extension to bounded domains with $C^2$ boundary. Lastly, we present the resulting fundamental gap estimates in terms of the first eigenvalue, the inradius and the diameter. Finally, Section \ref{sec proofs} contains the proofs of all results in this paper. 

\section{Preliminary results}\label{sec preliminaries}

In this section, we recall the auxiliary results used in the proofs. We begin with the definition of the $C_{p}$-functional and remainder estimates for the Poincar\'e inequality, followed by the weighted Poincar\'e inequality. Lastly, we recall the eigenvalue bounds by Brasco \cite{brasco2018principal} in terms of the inradius that enter the fundamental gap estimate.

We start with the definition of the $C_{p}$-functional, which will be used during the proof of the optimality of the exponent.

\begin{defn}
Let $1<p<\infty$. Then, for $\xi,\eta\in\mathbb{C}^{n}$, we define
\begin{align}\label{cp formula}
C_p(\xi,\eta):=|\xi|^p-|\xi-\eta|^p-p|\xi-\eta|^{p-2}\textnormal{Re}\left\langle\xi-\eta,\eta\right\rangle\geq0.
\end{align}
\end{defn}

The gradient stability argument begins with an improved Poincar\'e inequality. Such an inequality measures the variation of $u/u_{1}$, which vanishes precisely when $u$ is a multiple of $u_{1}$. In the degenerate case, the remainder controls this variation through its weighted $L^{p}$ norm.

\begin{thm}[{\cite{yessirkegenov2026stability, apseit2025sharp}}]\label{thm yz26}
Let $p,n\geq 2$ and $\Omega\subset\mathbb{R}^n$ be an open bounded domain. Then, for every $u\in W_{0}^{1,p}(\Omega)$, we have
\begin{align}\label{stability}
&\int_{\Omega}|\nabla u|^{p}dx-\lambda_{1}(\Omega,p)\int_{\Omega}|u|^{p}dx\geq c_{\mathrm{deg}}(p)\int_{\Omega}
\left|u_{1}\nabla\left(\frac{u}{u_{1}}\right)\right|^{p}dx.
\end{align}
\end{thm}

One can use the following bounds to estimate the value of $c_{\mathrm{deg}}(p)$ in terms of $p$ only.

\begin{prop}[{\cite[Proposition~2.10]{yessirkegenov2026stability}}]
For every $p\geq 2$, we have
\begin{align}\label{lower bound cp}
\frac{1}{2^{p-2}}\leq c_{\mathrm{deg}}(p)\leq\frac{p-1}{2^{p-2}}.
\end{align}
Moreover, $c_{\mathrm{deg}}(p)\to 0$ as $p\to\infty$.
\end{prop}

For $1<p<2$, the corresponding remainder has a weighted quadratic form. The following estimate provides the starting point for the singular case.

\begin{thm}[\cite{apseit2025sharp}]\label{thm poinc 1p2}
Let $1<p<2$, $n\geq 2$ and let $\Omega\subset\mathbb{R}^{n}$ be an open bounded domain. Then, for every $u\in W_{0}^{1,p}(\Omega)$, we have
\begin{align}
&\int_{\Omega}|\nabla u|^{p}dx-\lambda_{1}(\Omega,p)\int_{\Omega}|u|^{p}dx\notag\\
&\qquad\geq \underline{c}_{\mathrm{sing}}(p)\int_{\Omega}
\left(|\nabla u|+\left|\frac{u}{u_{1}}\nabla u_{1}\right|\right)^{p-2}
\left|u_1\nabla\left(\frac{u}{u_{1}}\right)\right|^{2}dx,
\label{eq singular remainder preliminaries}
\end{align}
where
\begin{align*}
\underline{c}_{\mathrm{sing}}(p)
:=\inf_{s^{2}+t^{2}>0}
\frac{\left(t^{2}+s^{2}+2s+1\right)^{p/2}-1-ps}
{\left(\sqrt{t^{2}+s^{2}+2s+1}+1\right)^{p-2}(s^{2}+t^{2})}\in\left(0,\frac{p(p-1)}{2^{p-1}}\right].
\end{align*}
\end{thm}

As for $c_{\mathrm{deg}}(p)$, we provide explicit bounds for $\underline{c}_{\mathrm{sing}}(p)$ depending only on $p$, which allow us to avoid the infimum defining $\underline{c}_{\mathrm{sing}}(p)$ when estimating the stability constant.

\begin{prop}\label{prop for expl 1<p<2}
For every $1<p<2$, we have
\begin{align}\label{eq singular constant bounds}
\left(\frac{3-p}{4-p}\right)^{3-p}\frac{p(p-1)}{2^{p-2}}
\leq \underline{c}_{\mathrm{sing}}(p)\leq\frac{p(p-1)}{2^{p-1}}.
\end{align}
\end{prop}

We also use the following upper bound for the Poincar\'e deficit in the singular case. This estimate will be used to obtain an upper bound in the gradient stability inequality.

\begin{thm}[\cite{apseit2025sharp}]\label{thm poinc 1p2 upper}
Let $1<p<2$, $n\geq 2$ and let $\Omega\subset\mathbb{R}^{n}$ be an open bounded domain. Then, for every $u\in W_{0}^{1,p}(\Omega)$, we have
\begin{align}
&\int_{\Omega}|\nabla u|^{p}dx-\lambda_{1}(\Omega,p)\int_{\Omega}|u|^{p}dx\notag\\
&\qquad\leq \overline{c}_{\mathrm{sing}}(p)\int_{\Omega}
\left(|\nabla u|+\left|\frac{u}{u_{1}}\nabla u_{1}\right|\right)^{p-2}
\left|u_1\nabla\left(\frac{u}{u_{1}}\right)\right|^{2}dx,
\label{eq singular remainder upper preliminaries}
\end{align}
where
\begin{align}\label{eq singular upper constant}
\overline{c}_{\mathrm{sing}}(p)
:=\sup_{s^{2}+t^{2}>0}
\frac{\left(t^{2}+s^{2}+2s+1\right)^{p/2}-1-ps}
{\left(\sqrt{t^{2}+s^{2}+2s+1}+1\right)^{p-2}(s^{2}+t^{2})}\in\left[\frac{p}{2^{p-1}},+\infty\right).
\end{align}
\end{thm}

The bounds (\ref{stability}), (\ref{eq singular remainder preliminaries}) already suggest the stability exponent $\max\{2,p\}$. Indeed, the factor $u_{1}\nabla(u/u_{1})$ appears to the power $p$ when $p\geq2$ and to the power $2$, with an additional weight, when $1<p<2$. Since these remainders bound the Poincar\'e deficit from below, they hint at the corresponding powers in the gradient stability estimate. The inequalities (\ref{stability}),
(\ref{eq singular remainder preliminaries}) and
(\ref{eq singular remainder upper preliminaries}) are obtained
by the usual argument based on integration by parts and the first eigenfunction equation. The quantitative remainders, however, rely on the bounds for the $C_p$-term established by Cazacu, Krej{\v{c}}i{\v{r}}{\'\i}k, Lam and Laptev
\cite{cazacu2024hardy} for $p\geq2$ and by Chen and Tang \cite{CT24} for $1<p<2$. In the case $p\geq 2$, the sharp form of the constant $c_{\mathrm{deg}}(p)$ in \eqref{eq degenerate constant} was obtained by Huang and Tong \cite{huang2025lp}. One can also find improvements of the form (\ref{stability}) and (\ref{eq singular remainder preliminaries}) in the work of Bobkov and Kolonitskii \cite{bobkov2023improved}, where the explicit constant comes from the hidden convexity \cite[Eq. (2.10)]{BrascoPrinariZagati2022}.

For $p>1$ and $\Omega\in\mathcal{C}_{n}$, the positive first eigenfunction $u_{1}$ is log-concave \cite{sakaguchi1987concavity}. Hence $u_{1}^{p}$ is a positive log-concave weight, which permits the use of the following weighted Poincar\'e inequality with a constant bounded in terms of the diameter of $\Omega$.

\begin{thm}[{\cite[Theorem 1.1]{ferone2012remark}}]\label{thm log-convex stab}
Let $1<p<\infty$, $n\geq 2$ and $\Omega\in\mathcal{C}_{n}$, and let $\omega$ be a positive log-concave weight on $\Omega$. Then there exists a constant $C>0$, depending on $\Omega$, $p$ and $\omega$, such that every Lipschitz function $u$ with
\begin{align}\label{eq weighted mean condition}
\int_{\Omega}|u|^{p-2}u\omega dx=0
\end{align}
satisfies
\begin{align}\label{eq weighted poincare}
\int_{\Omega}|\nabla u|^{p}\omega dx
\geq C\inf_{t\in\mathbb{R}}
\int_{\Omega}|u-t|^{p}\omega dx.
\end{align}
Moreover,
\begin{align}\label{eq weighted poincare coefficient}
C\geq
\left(\frac{\pi_{p}}{\mathrm{diam}(\Omega)}\right)^{p}.
\end{align}
\end{thm}

In the proofs, we apply (\ref{eq weighted poincare}) together with (\ref{eq weighted poincare coefficient}) to $u/u_{1}-c$ with weight $u_{1}^{p}$. Choosing $c$ to minimize $\int_{\Omega}|u-cu_{1}|^{p}dx$ gives the weighted mean condition (\ref{eq weighted mean condition}). This links the weighted derivative in the remainder estimates to the distance from the set of optimizers.

The resulting stability estimates still involve $\lambda_{1}(\Omega,p)$. To obtain the geometric fundamental gap bound, we will use both sides of the following estimate to express this dependence in terms of the inradius.

\begin{thm}[{\cite[Theorem~1.2 and Remark~1.4]{brasco2018principal}}]\label{thm brasco inradius}
Let $1<p<\infty$, $n\geq 2$ and $\Omega\in\mathcal{C}_{n}$. Then
\begin{align}\label{eq brasco inradius bounds}
\left(\frac{\pi_{p}}{2r_{\Omega}}\right)^{p}
\leq\lambda_{1}(\Omega,p)
\leq\frac{\lambda_{1}(B_{1}(0),p)}{r_{\Omega}^{p}}.
\end{align}
Moreover, both constants are sharp.
\end{thm}

\section{Main results}\label{sec main results}

In this section, we state the gradient stability inequality for the $L^{p}$-Poincar\'e inequality with exponent $\max\{2,p\}$. We also give explicit bounds for the stability constants and derive a fundamental gap estimate together with its geometric implication in terms of the inradius and diameter.

\begin{thm}\label{thm gradient stability}
Let $1<p<\infty$, $n\geq 2$ and $\Omega\in\mathcal{C}_{n}$. Let $\rho>p$ be the unique solution of
\begin{align}\label{eq rho}
\rho^{p-1}(\rho-p)-\lambda_{1}(\Omega,p)
\left(\frac{\mathrm{diam}(\Omega)}{\pi_{p}}\right)^{p}=0.
\end{align}
Then, for every $0\neq u\in W_{0}^{1,p}(\Omega)$, we have
\begin{align}\label{eq gradient stability unified}
1-\lambda_{1}(\Omega,p)\frac{\norm{u}_{L^{p}(\Omega)}^{p}}{\norm{\nabla u}_{L^{p}(\Omega)}^{p}}
\geq C_{\Omega,p}
\inf_{c\in\mathbb{R}}\left(\frac{\norm{\nabla(u-cu_{1})}_{L^{p}(\Omega)}}{\norm{\nabla u}_{L^{p}(\Omega)}}\right)^{\alpha},
\end{align}
where $\alpha:=\max\{2,p\}$ and
\begin{align}\label{eq stability constant}
C_{\Omega,p}:=\frac{1}{(1+\rho)^{\max\{2,p\}}}\times
\begin{cases}
c_{\mathrm{deg}}(p),&p\geq 2,\\[1ex]
\underline{c}_{\mathrm{sing}}(p)\left(\dfrac{p-1}{2p-1}\right)^{2-p},&1<p<2.
\end{cases}
\end{align}
Moreover, $\alpha$ is optimal.
\end{thm}

\begin{rem}\label{rem explicit stability}
Under the assumptions of Theorem \ref{thm gradient stability}, we can obtain more explicit bounds for the stability constants. Since $0<\rho-p<\rho$, the defining equation (\ref{eq rho}) gives
\begin{align}
(\rho-p)^{p}
&=(\rho-p)^{p-1}(\rho-p)\notag\\
&<\rho^{p-1}(\rho-p)\notag\\
&=\lambda_{1}(\Omega,p)
\left(\frac{\mathrm{diam}(\Omega)}{\pi_{p}}\right)^{p}.
\label{eq rho power bound}
\end{align}
Taking the $p$-th root in (\ref{eq rho power bound}) and adding $p$ gives
\begin{align}\label{eq rho upper bound}
\rho<p+\lambda_{1}(\Omega,p)^{1/p}\frac{\mathrm{diam}(\Omega)}{\pi_{p}}.
\end{align}
Substituting (\ref{eq rho upper bound}) into the definition (\ref{eq stability constant}), we obtain
\begin{align}\label{eq stability constant lower bound}
C_{\Omega,p}>\frac{1}{
\left(p+1+\lambda_{1}(\Omega,p)^{1/p}\frac{\mathrm{diam}(\Omega)}{\pi_{p}}\right)^{\max\{2,p\}}}\times
\begin{cases}
c_{\mathrm{deg}}(p),&p\geq 2,\\[1ex]
\underline{c}_{\mathrm{sing}}(p)\left(\dfrac{p-1}{2p-1}\right)^{2-p},&1<p<2.
\end{cases}
\end{align}
The lower bounds (\ref{eq stability constant lower bound}) give more explicit constants in (\ref{eq gradient stability unified}).
\end{rem}

We next explain how the singular case of Theorem \ref{thm gradient stability} yields the conjectured inequality (\ref{eq:8.1}) of Fleckinger-Pell\'e and Tak\'a\v{c} on open bounded convex domains. Let $1<p<2$ and $\Omega\in\mathcal{C}_{n}$, normalize $u_{1}$ by $\norm{u_{1}}_{L^{p}(\Omega)}=1$ and write $p'=p/(p-1)$. The first eigenfunction is bounded, so $u_{1}\in L^{p'}(\Omega)$. Hence the coefficient $u^{\parallel}$ in (\ref{eq projection}) is well-defined for every $u\in W_{0}^{1,p}(\Omega)$. For every $c\in\mathbb{R}$, H\"older's and $L^{p}$-Poincar\'e inequalities give
\begin{align}
|u^{\parallel}-c|
&=\frac{\left|\int_{\Omega}(u-cu_{1})u_{1}dx\right|}{\norm{u_{1}}_{L^{2}(\Omega)}^{2}}\notag\\
&\leq\frac{\norm{u_{1}}_{L^{p'}(\Omega)}}{\norm{u_{1}}_{L^{2}(\Omega)}^{2}}
\norm{u-cu_{1}}_{L^{p}(\Omega)}\notag\\
&\leq\frac{\norm{u_{1}}_{L^{p'}(\Omega)}}{\lambda_{1}(\Omega,p)^{1/p}\norm{u_{1}}_{L^{2}(\Omega)}^{2}}
\norm{\nabla(u-cu_{1})}_{L^{p}(\Omega)}.\label{eq projection coefficient bound}
\end{align}
Since $u^{\perp}=(u-cu_{1})-(u^{\parallel}-c)u_{1}$ and $\norm{\nabla u_{1}}_{L^{p}(\Omega)}=\lambda_{1}(\Omega,p)^{1/p}$, the triangle inequality together with (\ref{eq projection coefficient bound}) therefore yields
\begin{align}
\norm{\nabla u^{\perp}}_{L^{p}(\Omega)}
&\leq\norm{\nabla(u-cu_{1})}_{L^{p}(\Omega)}
+|u^{\parallel}-c|\norm{\nabla u_{1}}_{L^{p}(\Omega)}\notag\\
&\leq\left(1+\frac{\norm{u_{1}}_{L^{p'}(\Omega)}}{\norm{u_{1}}_{L^{2}(\Omega)}^{2}}\right)
\norm{\nabla(u-cu_{1})}_{L^{p}(\Omega)}.
\label{eq projected gradient comparison}
\end{align}
Taking the infimum over $c\in\mathbb{R}$ gives the required comparison between the projected gradient and the distance to $\mathbb{R}u_{1}$. Using (\ref{eq projected gradient comparison}) in (\ref{eq gradient stability unified}) gives the conjectured estimate with $\norm{\nabla u}_{L^{p}(\Omega)}^{-(2-p)}$ on the right-hand side. 

\begin{thm}\label{thm takac conjecture}
Let $1<p<2$, $n\geq2$ and $\Omega\in\mathcal{C}_{n}$. Let $u_{1}$ be the positive first eigenfunction normalized by $\norm{u_{1}}_{L^{p}(\Omega)}=1$ and let $p'=p/(p-1)$. Then, for every $0\neq u\in W_{0}^{1,p}(\Omega)$, we have
\begin{multline*}
\int_{\Omega}|\nabla u|^{p}dx-\lambda_{1}(\Omega,p)\int_{\Omega}|u|^{p}dx\\
\geq C_{\Omega,p}
\left(\frac{\norm{u_{1}}_{L^{2}(\Omega)}^{2}}
{\norm{u_{1}}_{L^{2}(\Omega)}^{2}+\norm{u_{1}}_{L^{p'}(\Omega)}}\right)^{2}
\norm{\nabla u}_{L^{p}(\Omega)}^{-(2-p)}
\norm{\nabla u^{\perp}}_{L^{p}(\Omega)}^{2},
\end{multline*}
where $C_{\Omega,p}$ is defined in (\ref{eq stability constant}).
\end{thm}

One can obtain the full statement of the conjecture (\ref{eq:8.1}) without assuming that $\Omega$ is convex at the cost of losing the explicit geometric formula for the constant.

\begin{thm}\label{thm takac conjecture general}
Let $1<p<2$, $n\geq2$ and let $\Omega\subset\mathbb{R}^{n}$ be a bounded domain with $C^{2}$ boundary. Let $\norm{u_{1}}_{L^{p}(\Omega)}=1$. Then there exists a constant $c>0$, depending on $\Omega$ and $p$, such that, for every $u\in W_{0}^{1,p}(\Omega)$, we have
\begin{align*}
\int_{\Omega}|\nabla u|^{p}dx-\lambda_{1}(\Omega,p)\int_{\Omega}|u|^{p}dx\geq c\norm{\nabla u}_{L^{p}(\Omega)}^{-(2-p)}
\norm{\nabla u^{\perp}}_{L^{p}(\Omega)}^{2}.
\end{align*}
\end{thm}

The lower stability estimates above are complemented by an upper bound in the singular case. Unlike the lower bound in Theorem \ref{thm gradient stability}, this estimate does not require convexity and has the optimal exponent $p$.

\begin{thm}\label{thm upper gradient stability}
Let $1<p<2$, $n\geq2$ and let $\Omega\subset\mathbb{R}^{n}$ be an open bounded domain. Then, for every $0\neq u\in W_{0}^{1,p}(\Omega)$, we have
\begin{align}\label{eq normalized upper gradient stability}
1-\lambda_{1}(\Omega,p)
\frac{\norm{u}_{L^{p}(\Omega)}^{p}}{\norm{\nabla u}_{L^{p}(\Omega)}^{p}}\leq\overline{c}_{\mathrm{sing}}(p)
\left(\frac{2p-1}{p-1}\right)^{p}
\inf_{c\in\mathbb{R}}
\left(\frac{\norm{\nabla(u-cu_1)}_{L^{p}(\Omega)}}{\norm{\nabla u}_{L^{p}(\Omega)}}\right)^{p}.
\end{align}
Moreover, the exponent $p$ is optimal.
\end{thm}

If we additionally assume that $\Omega\in\mathcal{C}_{n}$ in Theorem \ref{thm upper gradient stability}, we can combine it with Theorem \ref{thm gradient stability} to obtain the following two-sided estimate for the normalized Poincar\'e deficit.

\begin{cor}
Let $1<p<2$, $n\geq2$ and $\Omega\in\mathcal{C}_{n}$. For every $0\neq u\in W_{0}^{1,p}(\Omega)$, define the normalized Poincar\'e deficit by
\begin{align*}
\delta_{\Omega,p}(u):=1-\lambda_{1}(\Omega,p)
\frac{\norm{u}_{L^{p}(\Omega)}^{p}}{\norm{\nabla u}_{L^{p}(\Omega)}^{p}}.
\end{align*}
Then,
\begin{multline*}
\frac{\underline{c}_{\mathrm{sing}}(p)}{(1+\rho)^{\max\{2,p\}}}\left(\dfrac{p-1}{2p-1}\right)^{2-p}\inf_{c\in\mathbb{R}}
\left(\frac{\norm{\nabla(u-cu_{1})}_{L^{p}(\Omega)}}{\norm{\nabla u}_{L^{p}(\Omega)}}\right)^{2}
\leq\delta_{\Omega,p}(u)\\
\leq\overline{c}_{\mathrm{sing}}(p)
\left(\frac{2p-1}{p-1}\right)^{p}
\inf_{c\in\mathbb{R}}
\left(\frac{\norm{\nabla(u-cu_{1})}_{L^{p}(\Omega)}}{\norm{\nabla u}_{L^{p}(\Omega)}}\right)^{p},
\end{multline*}
Moreover, the exponents $2$ in the lower bound and $p$ in the upper bound are optimal.
\end{cor}

\begin{rem}
The two-sided estimate (\ref{eq two sided gradient stability intro})
gives a global quantitative characterization of almost optimizers. Indeed, for any sequence of nonzero functions
$u_k\in W_0^{1,p}(\Omega)$, we have
\begin{align*}
\delta_{\Omega,p}(u_k)\to0
\quad\Longleftrightarrow\quad
\inf_{c\in\mathbb{R}}
\frac{\norm{\nabla(u_k-cu_1)}_{L^p(\Omega)}}
{\norm{\nabla u_k}_{L^p(\Omega)}}\to0.
\end{align*}
The optimal upper exponent $p$ agrees with the Sobolev upper
bound of Ignat and Zuazua
\cite[Corollary~3.5 and Remark~3.6]{ignat2025optimal},
whose optimal exponent is $\min\{2,p\}$.
Their estimate holds when the normalized gradient distance
to the Sobolev optimizers is sufficiently small.
Here, the Poincar\'e estimate holds globally, so it applies
to every term of a minimizing sequence, including those
far from the optimizer set.
Moreover, our upper constant is explicit and depends only
on $p$, while the lower constant has the explicit geometric
dependence given in (\ref{eq stability constant}).
\end{rem}

We now turn to the fundamental gap on bounded convex domains. Applying the gradient stability estimate (\ref{eq gradient stability unified}) with the mountain-pass characterization (\ref{eigenvalue second}) of the second eigenvalue yields the following explicit lower bound for every $1<p<\infty$.

\begin{thm}\label{thm fundamental gap}
Let $1<p<\infty$, $n\geq 2$ and $\Omega\in\mathcal{C}_{n}$. Let $C_{\Omega,p}$ be the constant in (\ref{eq stability constant}). Then,
\begin{align}\label{eq unified spectral gap}
\Gamma(\Omega,p)\geq\frac{C_{\Omega,p}}{1-C_{\Omega,p}}\lambda_{1}(\Omega,p),
\end{align}
where $0<C_{\Omega,p}<1$. Equivalently,
\begin{align*}
\Gamma(\Omega,p)\geq C_{\Omega,p}\lambda_{2}(\Omega,p).
\end{align*}
In particular, we have the explicit bound
\begin{align}
\Gamma(\Omega,p)\geq{}&\frac{\lambda_{1}(\Omega,p)}{
\left(p+1+\lambda_{1}(\Omega,p)^{1/p}\frac{\mathrm{diam}(\Omega)}{\pi_{p}}\right)^{\max\{2,p\}}}
\notag\\&\times
\begin{cases}
c_{\mathrm{deg}}(p),&p\geq2,\\[1ex]
\underline{c}_{\mathrm{sing}}(p)\left(\dfrac{p-1}{2p-1}\right)^{2-p},&1<p<2.
\end{cases}
\label{eq explicit spectral gap}
\end{align}
\end{thm}

Together with the inradius bounds of Brasco \cite[Theorem~1.2 and Remark~1.4]{brasco2018principal}, our Theorem \ref{thm fundamental gap} leads to a bound depending only on the geometry of the domain.

\begin{cor}\label{cor inradius fundamental gap}
Under the assumptions of Theorem \ref{thm fundamental gap}, let $r_{\Omega}$ be the inradius of $\Omega$. Then,
\begin{align}\label{eq unified geometric gap}
\Gamma(\Omega,p)\geq K_{n,p}
\frac{r_{\Omega}^{\max\{2,p\}-p}}{\mathrm{diam}(\Omega)^{\max\{2,p\}}},
\end{align}
where
\begin{align}
K_{n,p}:=\frac{\pi_{p}^{p}}{
2^{p}\left(\frac{p+1}{2}+\frac{\lambda_{1}(B_{1}(0),p)^{1/p}}{\pi_{p}}\right)^{\max\{2,p\}}}
\times
\begin{cases}
c_{\mathrm{deg}}(p),&p\geq 2,\\[1ex]
\underline{c}_{\mathrm{sing}}(p)\left(\dfrac{p-1}{2p-1}\right)^{2-p},&1<p<2
\end{cases}
\label{eq geometric gap constant}
\end{align}
and $B_{1}(0)$ denotes the unit ball in $\mathbb{R}^{n}$.
\end{cor}

\begin{rem}
In the singular case $1<p<2$, Chen and Hauer \cite[Theorems~1.2 and~1.3]{chen2026fundamental} constructed smooth convex domains collapsing to a segment whose fundamental gaps tend to zero, ruling out a positive lower bound depending only on the diameter. Corollary \ref{cor inradius fundamental gap} complements this result by giving an explicit positive lower bound involving both the inradius and the diameter. The factor $r_{\Omega}^{2-p}$ allows this lower bound to tend to zero under collapse at fixed diameter.
\end{rem}

\section{Proofs}\label{sec proofs}

In this section, we prove the results stated in Section \ref{sec main results}. We first prove the explicit bounds for $\underline{c}_{\mathrm{sing}}(p)$ in Proposition \ref{prop for expl 1<p<2}, then the gradient stability inequality, treating the degenerate and singular cases separately. We then use the mountain-pass characterization of the second eigenvalue to prove the fundamental gap estimate and apply the bounds of Brasco to obtain its inradius corollary.

\begin{proof}[\textbf{Proof of Proposition \ref{prop for expl 1<p<2}}]
Let $1<p<2$. Fix $s,t\in\mathbb{R}$ with $s^2+t^2>0$, and define
\begin{align*}
r=\sqrt{(1+s)^2+t^2},
\qquad f(u)=\left((1+us)^2+u^2t^2\right)^{p/2},
\qquad 0\leq u\leq1.
\end{align*}
Then $f(0)=1$, $f(1)=r^p$ and $f'(0)=ps$. By Cauchy-Schwarz,
\begin{align*}
\left(s+u(s^2+t^2)\right)^2
\leq(s^2+t^2)\left((1+us)^2+u^2t^2\right).
\end{align*}
Since $p-2<0$, direct differentiation therefore gives
\begin{align}
f''(u)
&=p(s^2+t^2)\left((1+us)^2+u^2t^2\right)^{(p-2)/2}\notag\\
&\quad+p(p-2)\left(s+u(s^2+t^2)\right)^2
\left((1+us)^2+u^2t^2\right)^{(p-4)/2}\notag\\
&\geq p(p-1)(s^2+t^2)\left((1+us)^2+u^2t^2\right)^{(p-2)/2}.
\label{eq singular proof second derivative lower bound}
\end{align}
The derivative formula holds wherever $(1+us,ut)\neq(0,0)$. If this vector vanishes at $u_0\in(0,1]$, then $t=0$, $s=-1/u_0$ and $f(u)=|s|^p|u-u_0|^p$. Since $p-2>-1$, its second derivative is integrable near $u_0$ and $f'$ is absolutely continuous. Otherwise, $f\in C^2([0,1])$. Thus, Taylor's formula with integral remainder is valid in either case and gives
\begin{align*}
r^p-1-ps=\int_0^1(1-u)f''(u)\,du.
\end{align*}
The triangle inequality yields
\begin{align}\label{eq singular proof triangle inequality}
\sqrt{(1+us)^2+u^2t^2}
=\left|(1-u)(1,0)+u(1+s,t)\right|
\leq1-u+ur.
\end{align}
Since $p-2<0$, substituting the bounds (\ref{eq singular proof second derivative lower bound}) and (\ref{eq singular proof triangle inequality}) into Taylor's integral remainder formula gives
\begin{align}\label{eq singular proof remainder lower bound}
r^p-1-ps
\geq p(p-1)(s^2+t^2)\int_0^1(1-u)(1-u+ur)^{p-2}\,du.
\end{align}
To estimate this integral, we use the conjugate H\"older exponents $3-p$ and $(3-p)/(2-p)$ and the factorization
\begin{align*}
(1-u)^{1/(3-p)}
=\left((1-u)(1-u+ur)^{p-2}\right)^{1/(3-p)}
(1-u+ur)^{(2-p)/(3-p)}.
\end{align*}
H\"older's inequality gives
\begin{align}
\int_0^1(1-u)^{1/(3-p)}\,du
&\leq\left(\int_0^1(1-u)(1-u+ur)^{p-2}\,du\right)^{1/(3-p)}\notag\\
&\quad\times\left(\int_0^1(1-u+ur)\,du\right)^{(2-p)/(3-p)}.
\label{eq singular proof holder inequality}
\end{align}
The integrals are finite also when $r=0$ since then $(1-u)(1-u+ur)^{p-2}=(1-u)^{p-1}$ for $u<1$. Direct integration gives
\begin{align*}
\int_0^1(1-u)^{1/(3-p)}\,du=\frac{3-p}{4-p},
\qquad
\int_0^1(1-u+ur)\,du=\frac{1+r}{2}.
\end{align*}
Substituting these values into (\ref{eq singular proof holder inequality}) and rearranging, we find
\begin{align}\label{eq singular proof integral lower bound}
\int_0^1(1-u)(1-u+ur)^{p-2}\,du
\geq2^{2-p}\left(\frac{3-p}{4-p}\right)^{3-p}(1+r)^{p-2}.
\end{align}
Applying (\ref{eq singular proof integral lower bound}) in (\ref{eq singular proof remainder lower bound}) and dividing by $(1+r)^{p-2}(s^2+t^2)>0$, we obtain
\begin{align*}
\frac{r^p-1-ps}{(1+r)^{p-2}(s^2+t^2)}
\geq p(p-1)2^{2-p}\left(\frac{3-p}{4-p}\right)^{3-p}.
\end{align*}
Taking the infimum over $s,t$ proves the lower bound.
\end{proof}

\begin{proof}[\textbf{Proof of Theorem \ref{thm gradient stability}}]
Let $u\in C_{0}^{\infty}(\Omega)$ with $u\neq 0$. We first establish a gradient conversion estimate valid for every $p>1$. Choose $c\in\mathbb{R}$ to minimize 
\begin{align}\label{int cond}
\int_{\Omega}|u-cu_{1}|^{p}dx
\end{align}
and define
\begin{align*}
f:=\frac{u}{u_{1}}-c.
\end{align*}
Since $c$ minimizes (\ref{int cond}), we have
\begin{align*}
\int_{\Omega}|f|^{p-2}f u_{1}^{p}dx = 0.
\end{align*}
Since $u_{1}$ is positive and log-concave \cite{sakaguchi1987concavity}, so is $u_{1}^{p}$. The weighted Poincar\'e inequality in Theorem \ref{thm log-convex stab} therefore gives
\begin{align}\label{poincare weighted eq}
\left(\int_{\Omega}|\nabla f|^{p}|u_{1}|^{p}dx\right)^{\frac{1}{p}} \geq \frac{\pi_{p}}{\mathrm{diam}(\Omega)}\left(\int_{\Omega}|f|^{p}|u_{1}|^{p}dx\right)^{\frac{1}{p}}.
\end{align}
Next, we estimate the term $f\nabla u_{1}$. First, let us choose $\varphi=u_{1}|f|^{p}$ as a test function in (\ref{eq weak eigenfunction}) for the first eigenpair $(\lambda_{1}(\Omega,p),u_{1})$. By the product rule and the chain rule,
\begin{align}
\nabla\left(u_{1}|f|^{p}\right)
&=|f|^{p}\nabla u_{1}+u_{1}\nabla\left(|f|^{p}\right)\notag\\
&=|f|^{p}\nabla u_{1}+p u_{1}|f|^{p-2}f\nabla f.
\label{eq gradient test function derivative}
\end{align}
This test function is admissible since $u\in C_{0}^{\infty}(\Omega)$ and $u_{1}$ is positive and locally $C^{1}$ in $\Omega$. Both $f$ and $\nabla f$ are bounded and $f=-c$ outside the compact support of $u$. The gradient in (\ref{eq gradient test function derivative}) therefore belongs to $L^{p}(\Omega)$. Moreover, $u_{1}|f|^{p}-|c|^{p}u_{1}$ has compact support in $\Omega$, so $u_{1}|f|^{p}\in W_{0}^{1,p}(\Omega)$. Thus, the weak formulation (\ref{eq weak eigenfunction}) gives
\begin{align}\label{eq gradient test identity}
\int_{\Omega}|f|^{p}|\nabla u_{1}|^{p}dx
+p\int_{\Omega}u_{1}|f|^{p-2}f|\nabla u_{1}|^{p-2}
\langle\nabla u_{1},\nabla f\rangle dx
=\lambda_{1}(\Omega,p)\int_{\Omega}u_{1}^{p}|f|^{p}dx.
\end{align}
Moving the second integral on the left-hand side of (\ref{eq gradient test identity}) to the right-hand side gives
\begin{align}\label{eq gradient test rearranged}
\int_{\Omega}|f|^{p}|\nabla u_{1}|^{p}dx
=\lambda_{1}(\Omega,p)\int_{\Omega}u_{1}^{p}|f|^{p}dx-p\int_{\Omega}u_{1}|f|^{p-2}f|\nabla u_{1}|^{p-2}
\langle\nabla u_{1},\nabla f\rangle dx.
\end{align}
Applying the Cauchy-Schwarz inequality and then H\"older's inequality to (\ref{eq gradient test rearranged}), we get
\begin{align}\label{eq gradient testing holder}
\int_{\Omega}|f\nabla u_{1}|^{p}dx \leq \lambda_{1}(\Omega,p)\int_{\Omega}|u_{1}f|^{p}dx + p\left(\int_{\Omega}|f\nabla u_{1}|^{p}dx\right)^{\frac{p-1}{p}}\left(\int_{\Omega}|u_{1}\nabla f|^{p}dx\right)^{\frac{1}{p}}.
\end{align}
Let us denote
\begin{align*}
s &= \left(\int_{\Omega}|f\nabla u_{1}|^{p}dx\right)^{\frac{1}{p}},\\
b &= \left(\int_{\Omega}|u_{1}\nabla f|^{p}dx\right)^{\frac{1}{p}}.
\end{align*}
Using the weighted Poincar\'e estimate (\ref{poincare weighted eq}) to bound the first term on the right-hand side of (\ref{eq gradient testing holder}), we find
\begin{align}\label{eq gradient conversion scalar}
s^{p}\leq\lambda_{1}(\Omega,p)
\left(\frac{\mathrm{diam}(\Omega)}{\pi_{p}}\right)^{p}b^{p}
+pbs^{p-1}.
\end{align}
If $b=0$, then $s=0$, so the required gradient estimate is immediate. Suppose that $b>0$. Dividing (\ref{eq gradient conversion scalar}) by $b^{p}$ and rearranging, we obtain
\begin{align}\label{eq gradient conversion ratio}
\left(\frac{s}{b}\right)^{p-1}\left(\frac{s}{b}-p\right)
\leq\lambda_{1}(\Omega,p)
\left(\frac{\mathrm{diam}(\Omega)}{\pi_{p}}\right)^{p}.
\end{align}
Let us verify that (\ref{eq rho}) has a unique solution $\rho>p$. The function
\begin{align*}
t\mapsto t^{p-1}(t-p)-\lambda_{1}(\Omega,p)
\left(\frac{\mathrm{diam}(\Omega)}{\pi_{p}}\right)^{p}
\end{align*}
is continuous on $[p,\infty)$. At $t=p$, its value is
\begin{align*}
-\lambda_{1}(\Omega,p)
\left(\frac{\mathrm{diam}(\Omega)}{\pi_{p}}\right)^{p}<0.
\end{align*}
On the other hand,
\begin{align*}
\lim_{t\to+\infty}\left(t^{p-1}(t-p)-\lambda_{1}(\Omega,p)
\left(\frac{\mathrm{diam}(\Omega)}{\pi_{p}}\right)^{p}\right)=+\infty,
\end{align*}
so the function is positive for sufficiently large $t$. Thus, by the Intermediate Value Theorem, there exists a zero $\rho>p$. Moreover,
\begin{align*}
\frac{d}{dt}\left(t^{p-1}(t-p)-\lambda_{1}(\Omega,p)
\left(\frac{\mathrm{diam}(\Omega)}{\pi_{p}}\right)^{p}\right)
=pt^{p-2}(t-p+1)>0\quad\text{for }t\geq p,
\end{align*}
which shows that the function is strictly increasing on $[p,\infty)$ and proves uniqueness. If $s/b\leq p$, then $s/b<\rho$. Otherwise, (\ref{eq gradient conversion ratio}) and strict monotonicity on $[p,\infty)$ give $s/b\leq\rho$. Thus, in both cases,
\begin{align}\label{eq ground state gradient bound}
\left(\int_{\Omega}|f\nabla u_{1}|^{p}dx\right)^{\frac{1}{p}}
\leq\rho\left(\int_{\Omega}|u_{1}\nabla f|^{p}dx\right)^{\frac{1}{p}}.
\end{align}
Since $u-cu_{1}=u_{1}f$, the product rule, the triangle inequality and (\ref{eq ground state gradient bound}) imply
\begin{align}
\left(\int_{\Omega}|\nabla(u-cu_{1})|^{p}dx\right)^{\frac{1}{p}}
&\leq \left(\int_{\Omega}|u_{1}\nabla f|^{p}dx\right)^{\frac{1}{p}} + \left(\int_{\Omega}|f\nabla u_{1}|^{p}dx\right)^{\frac{1}{p}}\notag\\
&\leq (1+\rho)\left(\int_{\Omega}|u_{1}\nabla f|^{p}dx\right)^{\frac{1}{p}}.\label{eq gradient conversion}
\end{align}

\noindent\textbf{(i) Case $p\geq 2$.} By Theorem \ref{thm yz26} and (\ref{eq gradient conversion}), we get
\begin{align}
\int_{\Omega}|\nabla u|^{p}dx-\lambda_{1}(\Omega,p)\int_{\Omega}|u|^{p}dx&\geq c_{\mathrm{deg}}(p)\int_{\Omega}|u_{1}\nabla f|^{p}dx\notag\\
&\geq C_{\mathrm{deg}}(\Omega,p)\norm{\nabla(u-cu_{1})}_{L^{p}(\Omega)}^{p}\notag\\
&\geq C_{\mathrm{deg}}(\Omega,p)\inf_{a\in\mathbb{R}}\norm{\nabla(u-au_{1})}_{L^{p}(\Omega)}^{p},
\label{eq degenerate gradient stability}
\end{align}
where
\begin{align*}
C_{\mathrm{deg}}(\Omega,p):=\frac{c_{\mathrm{deg}}(p)}{(1+\rho)^{p}}.
\end{align*}

\noindent\textbf{(ii) Case $1<p<2$.} We choose $\varphi=\frac{|u|^{p}}{u_{1}^{p-1}}$ as a test function in (\ref{eq weak eigenfunction}) for the first eigenpair $(\lambda_{1}(\Omega,p),u_{1})$:
\begin{align}\label{eq singular test identity}
\int_{\Omega}|\nabla u_{1}|^{p-2}\left\langle\nabla u_{1},\nabla\left(\frac{|u|^{p}}{u_{1}^{p-1}}\right)\right\rangle dx
=\lambda_{1}(\Omega,p)\int_{\Omega}|u|^{p}dx.
\end{align}
Since $u$ has compact support and $u_{1}$ is positive and locally $C^{1}$ in $\Omega$, the test function $\varphi$ belongs to $C_{0}^{1}(\Omega)\subset W_{0}^{1,p}(\Omega)$. Expanding the gradient on the left-hand side of (\ref{eq singular test identity}), we get
\begin{align}
&p\int_{\Omega}|\nabla u_{1}|^{p-2}u_{1}^{1-p}|u|^{p-2}u\langle\nabla u, \nabla u_{1}\rangle dx - (p-1)\int_{\Omega}\left|\frac{u}{u_{1}}\right|^{p}|\nabla u_{1}|^{p}dx \notag\\
&= \lambda_{1}(\Omega,p)\int_{\Omega}|u|^{p}dx.\label{eq singular eigenfunction identity}
\end{align}
Applying the Cauchy-Schwarz inequality, we have
\begin{align}
&|\nabla u_{1}|^{p-2}u_{1}^{1-p}|u|^{p-2}u\langle\nabla u,\nabla u_{1}\rangle\notag\\
&\qquad\leq |\nabla u_{1}|^{p-2}u_{1}^{1-p}|u|^{p-1}\left|\langle\nabla u,\nabla u_{1}\rangle\right|\notag\\
&\qquad\leq u_{1}^{1-p}|u|^{p-1}|\nabla u_{1}|^{p-1}|\nabla u|
=\left|\frac{u}{u_{1}}\nabla u_{1}\right|^{p-1}|\nabla u|.
\label{eq singular pointwise bound}
\end{align}
Rearranging (\ref{eq singular eigenfunction identity}) and using (\ref{eq singular pointwise bound}), followed by H\"older's inequality, gives
\begin{align}
(p-1)\int_{\Omega}\left|\frac{u}{u_{1}}\nabla u_{1}\right|^{p}dx &\leq p\int_{\Omega}\left|\frac{u}{u_{1}}\nabla u_{1}\right|^{p-1}|\nabla u|dx - \lambda_{1}(\Omega,p)\int_{\Omega}|u|^{p}dx \notag\\&\leq p\left(\int_{\Omega}\left|\frac{u}{u_{1}}\nabla u_{1}\right|^{p}dx\right)^{\frac{p-1}{p}}\left(\int_{\Omega}|\nabla u|^{p}dx\right)^{\frac{1}{p}}.
\label{eq singular testing holder}
\end{align}
If $\int_{\Omega}\left|\frac{u}{u_1}\nabla u_1\right|^p dx=0$,
then its $1/p$-power is zero and hence is bounded by
$\frac{p}{p-1}\norm{\nabla u}_{L^p(\Omega)}$.
Otherwise, dividing \eqref{eq singular testing holder}
by $(p-1)\left(\int_{\Omega}
\left|\frac{u}{u_1}\nabla u_1\right|^p dx
\right)^{\frac{p-1}{p}}$ gives the same bound.
Thus, in both cases,
\begin{align}\label{bound grad u1 ratio}
\left(\int_{\Omega}\left|\frac{u}{u_{1}}\nabla u_{1}\right|^{p}dx\right)^{\frac{1}{p}} \leq \frac{p}{p-1}\left(\int_{\Omega}|\nabla u|^{p}dx\right)^{\frac{1}{p}}.
\end{align}
Let us define
\begin{align*}
W&:=|\nabla u|+\left|\frac{u}{u_{1}}\nabla u_{1}\right|,\\
V&:=u_{1}\nabla\left(\frac{u}{u_{1}}\right).
\end{align*}
Observe that
\[
|V|\leq |\nabla u|
+\left|\frac{u}{u_1}\nabla u_1\right|=W.
\]
Thus, $V=0$ wherever $W=0$. Since $p-2<0$,
we set $W^{p-2}|V|^2=0$ on $\{W=0\}$. Now, by Theorem \ref{thm poinc 1p2}, we have
\begin{align}\label{eq singular remainder}
\int_{\Omega}|\nabla u|^{p}dx-\lambda_{1}(\Omega,p)\int_{\Omega}|u|^{p}dx \geq \underline{c}_{\mathrm{sing}}(p)\int_{\Omega}W^{p-2}|V|^{2}dx.
\end{align}
Using H\"older's inequality with conjugate exponents $\frac{2}{p}$ and $\frac{2}{2-p}$, we have
\begin{align}\label{eq singular holder application}
\int_{\Omega}|V|^{p}dx = \int_{\Omega}\left(|V|^{p}W^{\frac{p(p-2)}{2}}\right)W^{\frac{p(2-p)}{2}}dx \leq \left(\int_{\Omega}|V|^{2}W^{p-2}dx\right)^{\frac{p}{2}}\left(\int_{\Omega}W^{p}dx\right)^{\frac{2-p}{2}}.
\end{align}
Rearranging (\ref{eq singular holder application}) gives
\begin{align}\label{eq singular holder bound}
\int_{\Omega}W^{p-2}|V|^{2}dx \geq \frac{\left(\int_{\Omega}|V|^{p}dx\right)^{\frac{2}{p}}}{\left(\int_{\Omega}W^{p}dx\right)^{\frac{2-p}{p}}} = \frac{\left(\int_{\Omega}\left|u_{1}\nabla\left(\frac{u}{u_{1}}\right)\right|^{p}dx\right)^{\frac{2}{p}}}{\left(\int_{\Omega}\left(|\nabla u| + \left|\frac{u}{u_{1}}\nabla u_{1}\right|\right)^{p}dx\right)^{\frac{2-p}{p}}}.
\end{align}
By the triangle inequality and (\ref{bound grad u1 ratio}), the $L^{p}$ norm occurring in the denominator of (\ref{eq singular holder bound}) is bounded by
\begin{align}
\left(\int_{\Omega}\left(|\nabla u| + \left|\frac{u}{u_{1}}\nabla u_{1}\right|\right)^{p}dx\right)^{\frac{1}{p}} &\leq \left(\int_{\Omega}|\nabla u|^{p}dx\right)^{\frac{1}{p}} + \left(\int_{\Omega}\left|\frac{u}{u_{1}}\nabla u_{1}\right|^{p}dx\right)^{\frac{1}{p}} \notag\\
&\leq \frac{2p-1}{p-1}\left(\int_{\Omega}|\nabla u|^{p}dx\right)^{\frac{1}{p}}.\label{eq singular weight bound}
\end{align}
Thus, putting (\ref{eq singular remainder}), (\ref{eq singular holder bound}) and (\ref{eq singular weight bound}) together, we get
\begin{multline}\label{interm grad 1p2}
\int_{\Omega}|\nabla u|^{p}dx-\lambda_{1}(\Omega,p)\int_{\Omega}|u|^{p}dx\\
\geq \underline{c}_{\mathrm{sing}}(p)\left(\frac{p-1}{2p-1}\right)^{2-p}
\left(\int_{\Omega}|\nabla u|^{p}dx\right)^{-\frac{2-p}{p}}
\left(\int_{\Omega}\left|u_{1}\nabla\left(\frac{u}{u_{1}}\right)\right|^{p}dx\right)^{\frac{2}{p}}.
\end{multline}

Since $\nabla f=\nabla(u/u_{1})$, the inequality (\ref{interm grad 1p2}) with (\ref{eq gradient conversion}) gives
\begin{multline}\label{eq singular gradient stability}
\int_{\Omega}|\nabla u|^{p}dx-\lambda_{1}(\Omega,p)\int_{\Omega}|u|^{p}dx\\
\geq C_{\mathrm{sing}}(\Omega,p)\norm{\nabla u}_{L^{p}(\Omega)}^{-(2-p)}
\norm{\nabla(u-cu_{1})}_{L^{p}(\Omega)}^{2}\\
\geq C_{\mathrm{sing}}(\Omega,p)\norm{\nabla u}_{L^{p}(\Omega)}^{-(2-p)}
\inf_{a\in\mathbb{R}}\norm{\nabla(u-au_{1})}_{L^{p}(\Omega)}^{2},
\end{multline}
where
\begin{align*}
C_{\mathrm{sing}}(\Omega,p):=\frac{\underline{c}_{\mathrm{sing}}(p)}{(1+\rho)^{2}}
\left(\frac{p-1}{2p-1}\right)^{2-p}.
\end{align*}

\noindent\textbf{(iii) Case $1<p<\infty$.} Dividing (\ref{eq degenerate gradient stability}) for $p\geq2$ and (\ref{eq singular gradient stability}) for $1<p<2$ by $\norm{\nabla u}_{L^{p}(\Omega)}^{p}$ and assuming $u\neq 0$, we obtain the common inequality
\begin{align*}
1-\lambda_{1}(\Omega,p)\frac{\norm{u}_{L^{p}(\Omega)}^{p}}{\norm{\nabla u}_{L^{p}(\Omega)}^{p}}
\geq C_{\Omega,p}
\inf_{c\in\mathbb{R}}\left(\frac{\norm{\nabla(u-cu_{1})}_{L^{p}(\Omega)}}{\norm{\nabla u}_{L^{p}(\Omega)}}\right)^{\max\{2,p\}}.
\end{align*}
Indeed, the exponent is $p$ when $p\geq 2$ and $(2-p)+p=2$ gives the exponent $2$ when $1<p<2$. Finally, the $L^{p}$ norms and the distance to $\mathbb{R}u_{1}$ in the gradient norm are continuous in $W_{0}^{1,p}(\Omega)$. Density therefore extends (\ref{eq gradient stability unified}) to every nonzero $u\in W_{0}^{1,p}(\Omega)$.

\vspace{1mm}

\noindent\textbf{(iv) Optimality of the exponent $\max\{2,p\}$.}
Define
\begin{align*}
\delta(u)&:=1-\lambda_1\frac{\norm{u}_{L^{p}(\Omega)}^p}{\norm{\nabla u}_{L^{p}(\Omega)}^p},\\
d(u)&:=\inf_{a\in\mathbb{R}}\norm{\nabla(u-au_1)}_{L^{p}(\Omega)}.
\end{align*}
To prove optimality, suppose that, for some
$0<\beta<\max\{2,p\}$, there exists a constant $C_*>0$,
independent of $u$, such that
\begin{align}\label{eq optimality smaller exponent}
\delta(u)\geq C_*
\left(\frac{d(u)}{\norm{\nabla u}_{L^{p}(\Omega)}}\right)^\beta
\end{align}
for every nonzero $u\in W_0^{1,p}(\Omega)$. We construct perturbations of $u_1$ that contradict (\ref{eq optimality smaller exponent}). Let $\norm{u_1}_{L^{p}(\Omega)}=1$. Since
\begin{align*}
\norm{\nabla u_1}_{L^{p}(\Omega)}^p
=\lambda_1\norm{u_1}_{L^{p}(\Omega)}^p=\lambda_1>0
\end{align*}
and $u_1\in C^1_{\mathrm{loc}}(\Omega)$, there exists
$x_0\in\Omega$ such that $\nabla u_1(x_0)\neq0$. By continuity, choose $r>0$ sufficiently small that
\begin{align}\label{eq optimality local gradient bounds}
0<m\leq|\nabla u_1|\leq M
\qquad\text{on }B_r(x_0).
\end{align}
Decreasing $r$ if necessary, we may also assume that
\begin{align*}
M^p|B_1(0)|r^n\leq\frac12\norm{\nabla u_1}_{L^{p}(\Omega)}^p.
\end{align*}
Now set $B:=B_r(x_0)$. Using
(\ref{eq optimality local gradient bounds}), we obtain
\begin{align*}
\norm{\nabla u_1}_{L^{p}(\Omega\setminus B)}^p
&=\norm{\nabla u_1}_{L^{p}(\Omega)}^p-\int_B|\nabla u_1|^p\,dx\\
&\geq\norm{\nabla u_1}_{L^{p}(\Omega)}^p-M^p|B_1(0)|r^n\\
&\geq\frac12\norm{\nabla u_1}_{L^{p}(\Omega)}^p.
\end{align*}
Consequently,
\begin{align}\label{eq optimality exterior gradient bound}
\norm{\nabla u_1}_{L^{p}(\Omega\setminus B)}
\geq2^{-1/p}\norm{\nabla u_1}_{L^{p}(\Omega)}>0.
\end{align}
Let $h\in C_0^\infty(\Omega)$ have support compactly contained in $B$. Then, $h=\nabla h=0$ on $\Omega\setminus B$. For every $a\in\mathbb{R}$, we therefore have
\begin{align*}
\norm{\nabla(u_1+h-au_1)}_{L^{p}(\Omega)}^p
&\geq\int_{\Omega\setminus B}|(1-a)\nabla u_1|^p\,dx\\
&=|1-a|^p\norm{\nabla u_1}_{L^{p}(\Omega\setminus B)}^p\\
&\geq\frac12|1-a|^p\norm{\nabla u_1}_{L^{p}(\Omega)}^p,
\end{align*}
where the last inequality follows from
(\ref{eq optimality exterior gradient bound}).
Taking the $p$-th root gives
\begin{align}\label{eq optimality coefficient bound}
|1-a|\norm{\nabla u_1}_{L^{p}(\Omega)}
\leq2^{1/p}\norm{\nabla(u_1+h-au_1)}_{L^{p}(\Omega)}.
\end{align}
Since
\begin{align*}
\nabla h=\nabla(u_1+h-au_1)-(1-a)\nabla u_1,
\end{align*}
the triangle inequality and
(\ref{eq optimality coefficient bound}) imply
\begin{align*}
\norm{\nabla h}_{L^{p}(\Omega)}
&\leq\norm{\nabla(u_1+h-au_1)}_{L^{p}(\Omega)}+|1-a|\norm{\nabla u_1}_{L^{p}(\Omega)}\\
&\leq(1+2^{1/p})\norm{\nabla(u_1+h-au_1)}_{L^{p}(\Omega)}.
\end{align*}
Taking the infimum over $a\in\mathbb{R}$ yields
\begin{align*}
\norm{\nabla h}_{L^{p}(\Omega)}\leq(1+2^{1/p})d(u_1+h).
\end{align*}
For the reverse bound, choose $a=1$:
\begin{align*}
d(u_1+h)
&=\inf_{a\in\mathbb{R}}\norm{\nabla(u_1+h-au_1)}_{L^{p}(\Omega)}\\
&\leq\norm{\nabla(u_1+h-u_1)}_{L^{p}(\Omega)}=\norm{\nabla h}_{L^{p}(\Omega)}.
\end{align*}
Thus, we have
\begin{align}\label{eq optimality distance comparison}
\frac{\norm{\nabla h}_{L^{p}(\Omega)}}{1+2^{1/p}}
\leq d(u_1+h)\leq\norm{\nabla h}_{L^{p}(\Omega)}.
\end{align}
Moreover,
\begin{align*}
\norm{\nabla(u_1+h)}_{L^{p}(\Omega)}
\geq\norm{\nabla u_1}_{L^{p}(\Omega\setminus B)}>0,
\end{align*}
so $u_1+h\neq0$ and its deficit is well defined. We next estimate this deficit. Recall the convexity remainder
\begin{align}\label{eq optimality cp formula}
C_p(\xi,\eta):=
|\xi|^p-|\xi-\eta|^p
-p|\xi-\eta|^{p-2}\langle\xi-\eta,\eta\rangle\geq0.
\end{align}
For scalar arguments, (\ref{eq optimality cp formula}) gives
\begin{align*}
C_p(u_1+h,h)=|u_1+h|^p-u_1^p-pu_1^{p-1}h.
\end{align*}
Taking $\psi=h$ in
(\ref{eq weak eigenfunction}), we obtain
\begin{align*}
\int_B|\nabla u_1|^{p-2}
\langle\nabla u_1,\nabla h\rangle\,dx
=\lambda_1\int_Bu_1^{p-1}h\,dx.
\end{align*}
Using this identity to cancel the linear terms, together with
\begin{align*}
\norm{\nabla u_1}_{L^{p}(\Omega)}^p-\lambda_1\norm{u_1}_{L^{p}(\Omega)}^p=0,
\end{align*}
we calculate
\begin{align}
&\norm{\nabla(u_1+h)}_{L^{p}(\Omega)}^p\delta(u_1+h) \nonumber
\\&=\int_B\left(
|\nabla u_1+\nabla h|^p-|\nabla u_1|^p
\right)\,dx
-\lambda_1\int_B\left(|u_1+h|^p-u_1^p\right)\,dx \nonumber
\\&=\int_BC_p(\nabla u_1+\nabla h,\nabla h)\,dx
-\lambda_1\int_BC_p(u_1+h,h)\,dx. \label{eq optimality deficit identity}
\end{align}
The scalar remainder is nonnegative by
(\ref{eq optimality cp formula}). Hence,
\begin{align}\label{eq optimality deficit upper bound}
0\leq\norm{\nabla(u_1+h)}_{L^{p}(\Omega)}^p\delta(u_1+h)\leq\int_BC_p(\nabla u_1+\nabla h,\nabla h)\,dx.
\end{align}

\smallskip
\noindent\textit{Case $1<p\leq2$.}
Fix a nonzero $\varphi\in C_0^\infty(B)$ and set
$u_t:=u_1+t\varphi$.
By (\ref{eq optimality distance comparison}),
\begin{align}\label{eq optimality fixed distance}
\frac{|t|\norm{\nabla\varphi}_{L^{p}(\Omega)}}{1+2^{1/p}}
\leq d(u_t)\leq|t|\norm{\nabla\varphi}_{L^{p}(\Omega)}.
\end{align}
Since $\varphi$ is nonzero and compactly supported,
$\norm{\nabla\varphi}_{L^{p}(\Omega)}>0$. Concavity then gives
\begin{align*}
|\nabla u_1+t\nabla\varphi|^p
&=\left(
|\nabla u_1|^2+
2t\langle\nabla u_1,\nabla\varphi\rangle+
t^2|\nabla\varphi|^2
\right)^{p/2}\\
&\leq|\nabla u_1|^p
+\frac p2|\nabla u_1|^{p-2}
\left(
2t\langle\nabla u_1,\nabla\varphi\rangle+
t^2|\nabla\varphi|^2
\right)\\
&=|\nabla u_1|^p+
pt|\nabla u_1|^{p-2}
\langle\nabla u_1,\nabla\varphi\rangle
+\frac p2|\nabla u_1|^{p-2}t^2|\nabla\varphi|^2.
\end{align*}
Consequently, (\ref{eq optimality cp formula}) yields
\begin{align}\label{eq cp estimate}
&C_p(\nabla u_1+t\nabla\varphi,t\nabla\varphi) \nonumber
\\&=|\nabla u_1+t\nabla\varphi|^p-|\nabla u_1|^p
-pt|\nabla u_1|^{p-2}
\langle\nabla u_1,\nabla\varphi\rangle \nonumber
\\&\leq\frac p2|\nabla u_1|^{p-2}t^2|\nabla\varphi|^2 \nonumber
\\&\leq\frac p2m^{p-2}t^2|\nabla\varphi|^2.
\end{align}
Here, the last inequality follows from
(\ref{eq optimality local gradient bounds}) and $p-2\leq0$.
Substituting into
(\ref{eq optimality deficit upper bound}), we obtain
\begin{align}\label{eq optimality fixed energy}
0\leq\norm{\nabla u_t}_{L^{p}(\Omega)}^p\delta(u_t)
\leq\frac p2m^{p-2}|t|^2\int_B|\nabla\varphi|^2\,dx.
\end{align}
For sufficiently small $|t|$, we have
\begin{align}\label{eq optimality fixed gradient bounds}
\frac12\norm{\nabla u_1}_{L^{p}(\Omega)}
\leq\norm{\nabla u_t}_{L^{p}(\Omega)}
\leq\frac32\norm{\nabla u_1}_{L^{p}(\Omega)}.
\end{align}
Dividing (\ref{eq optimality fixed energy}) by the lower
bound in (\ref{eq optimality fixed gradient bounds}) gives
\begin{align*}
\delta(u_t)
&\leq\frac{2^{p-1}p\,m^{p-2}}{\norm{\nabla u_1}_{L^{p}(\Omega)}^p}
\left(\int_B|\nabla\varphi|^2\,dx\right)|t|^2\\
&\leq C'|t|^2,
\end{align*}
where $C'>0$ is independent of $t$.
Similarly, (\ref{eq optimality fixed distance}) and
(\ref{eq optimality fixed gradient bounds}) imply
\begin{align*}
\frac{d(u_t)}{\norm{\nabla u_t}_{L^{p}(\Omega)}}
&\geq\frac{2\norm{\nabla\varphi}_{L^{p}(\Omega)}}
{3(1+2^{1/p})\norm{\nabla u_1}_{L^{p}(\Omega)}}|t|\\
&\geq c|t|,
\end{align*}
where $c>0$ is independent of $t$. Since $\beta<2$, applying
(\ref{eq optimality smaller exponent}) to $u_t$ gives
\begin{align*}
C_*c^\beta|t|^\beta
&\leq C_*\left(\frac{d(u_t)}{\norm{\nabla u_t}_{L^{p}(\Omega)}}\right)^\beta\\
&\leq\delta(u_t)\leq C'|t|^2.
\end{align*}
For nonzero $t$, dividing by $|t|^\beta$ yields
\begin{align*}
0<C_*c^\beta\leq C'|t|^{2-\beta}\longrightarrow0
\qquad\text{as }t\to0,
\end{align*}
a contradiction.

\smallskip
\noindent\textit{Case $p>2$.}
Fix a nonzero $\varphi\in C_0^\infty(B_1(0))$ and define
\begin{align*}
h_\varepsilon(x):=
\varepsilon\varphi\left(\frac{x-x_0}{\varepsilon}\right),
\qquad u_\varepsilon:=u_1+h_\varepsilon.
\end{align*}
For $0<\varepsilon<r$,
\begin{align*}
\operatorname{supp}h_\varepsilon
\Subset B_\varepsilon(x_0)\subset B,
\end{align*}
so $h_\varepsilon\in C_0^\infty(B)$. Direct differentiation gives \begin{align}\label{eq optimality scaled gradient}
\nabla h_\varepsilon(x)
=\nabla\varphi\left(\frac{x-x_0}{\varepsilon}\right).
\end{align}
Making the change of variables
$y=(x-x_0)/\varepsilon$, with $dx=\varepsilon^n\,dy$,
we obtain
\begin{align*}
\norm{\nabla h_\varepsilon}_{L^{p}(\Omega)}^p
&=\int_{B_\varepsilon(x_0)}
\left|\nabla\varphi\left(\frac{x-x_0}{\varepsilon}\right)\right|^p\,dx\\
&=\varepsilon^n\int_{B_1(0)}|\nabla\varphi(y)|^p\,dy.
\end{align*}
Taking the $p$-th root gives
\begin{align}\label{eq optimality scaled gradient norm}
\norm{\nabla h_\varepsilon}_{L^{p}(\Omega)}
=\varepsilon^{n/p}\norm{\nabla\varphi}_{L^{p}(B_1(0))}.
\end{align}
Similarly,
\begin{align*}
\norm{h_\varepsilon}_{L^{p}(\Omega)}^p
&=\int_{B_\varepsilon(x_0)}
\varepsilon^p
\left|\varphi\left(\frac{x-x_0}{\varepsilon}\right)\right|^p\,dx\\
&=\varepsilon^{n+p}\int_{B_1(0)}|\varphi(y)|^p\,dy.
\end{align*}
Hence,
\begin{align*}
\norm{h_\varepsilon}_{L^{p}(\Omega)}
=\varepsilon^{1+n/p}\norm{\varphi}_{L^{p}(B_1(0))}.
\end{align*}
Combining (\ref{eq optimality distance comparison}) and
(\ref{eq optimality scaled gradient norm}), we obtain
\begin{align}\label{eq optimality shrinking distance}
\frac{\varepsilon^{n/p}\norm{\nabla\varphi}_{L^{p}(B_1(0))}}
{1+2^{1/p}}
\leq d(u_\varepsilon)
\leq\varepsilon^{n/p}\norm{\nabla\varphi}_{L^{p}(B_1(0))}.
\end{align}
By (\ref{eq optimality local gradient bounds}) and
(\ref{eq optimality scaled gradient}),
\begin{align*}
|\nabla u_1|\leq M,\qquad
|\nabla h_\varepsilon|
\leq\norm{\nabla\varphi}_{L^{\infty}(B_1(0))}
\qquad\text{on }B.
\end{align*}
In the same way as in (\ref{eq cp estimate}), we have
\begin{align*}
&C_p(\nabla u_1+\nabla h_\varepsilon,\nabla h_\varepsilon)\\
&=|\nabla u_1+\nabla h_\varepsilon|^p-|\nabla u_1|^p
-p|\nabla u_1|^{p-2}
\langle\nabla u_1,\nabla h_\varepsilon\rangle\\
&\leq\left(|\nabla u_1|+|\nabla h_\varepsilon|\right)^p
+p|\nabla u_1|^{p-1}|\nabla h_\varepsilon|\\
&\leq\left(M+\norm{\nabla\varphi}_{L^{\infty}(B_1(0))}\right)^p
+pM^{p-1}\norm{\nabla\varphi}_{L^{\infty}(B_1(0))}.
\end{align*}
This bound is independent of $\varepsilon$. Since
\begin{align*}
|B_\varepsilon(x_0)|=|B_1(0)|\varepsilon^n,
\end{align*}
we obtain
\begin{align*}
&\int_B
C_p(\nabla u_1+\nabla h_\varepsilon,\nabla h_\varepsilon)\,dx\\
&\leq\left[
\left(M+\norm{\nabla\varphi}_{L^{\infty}(B_1(0))}\right)^p
+pM^{p-1}\norm{\nabla\varphi}_{L^{\infty}(B_1(0))}
\right]|B_1(0)|\varepsilon^n.
\end{align*}
Consequently,
(\ref{eq optimality deficit upper bound}) gives
\begin{align}\label{eq optimality shrinking energy}
0\leq\norm{\nabla u_\varepsilon}_{L^{p}(\Omega)}^p\delta(u_\varepsilon)
\leq C'\varepsilon^n,
\end{align}
where $C'>0$ is independent of $\varepsilon$. For sufficiently small $\varepsilon>0$, we have
\begin{align}\label{eq optimality shrinking gradient bounds}
\frac12\norm{\nabla u_1}_{L^{p}(\Omega)}
\leq\norm{\nabla u_\varepsilon}_{L^{p}(\Omega)}
\leq\frac32\norm{\nabla u_1}_{L^{p}(\Omega)}.
\end{align}
Dividing (\ref{eq optimality shrinking energy}) by the
lower bound in
(\ref{eq optimality shrinking gradient bounds}) and
adjusting the constant independently of $\varepsilon$,
we obtain
\begin{align*}
\delta(u_\varepsilon)\leq C'\varepsilon^n.
\end{align*}
Similarly, (\ref{eq optimality shrinking distance}) and
(\ref{eq optimality shrinking gradient bounds}) imply
\begin{align*}
\frac{d(u_\varepsilon)}{\norm{\nabla u_\varepsilon}_{L^{p}(\Omega)}}
&\geq\frac{2\norm{\nabla\varphi}_{L^{p}(B_1(0))}}
{3(1+2^{1/p})\norm{\nabla u_1}_{L^{p}(\Omega)}}\varepsilon^{n/p}\\
&\geq c\varepsilon^{n/p},
\end{align*}
where $c>0$ is independent of $\varepsilon$. Since $\beta<p$, applying
(\ref{eq optimality smaller exponent}) to
$u_\varepsilon$ gives
\begin{align*}
C_*c^\beta\varepsilon^{n\beta/p}
&\leq C_*\left(
\frac{d(u_\varepsilon)}{\norm{\nabla u_\varepsilon}_{L^{p}(\Omega)}}
\right)^\beta\\
&\leq\delta(u_\varepsilon)\leq C'\varepsilon^n.
\end{align*}
Dividing by $\varepsilon^{n\beta/p}>0$, we conclude that
\begin{align*}
0<C_*c^\beta
\leq C'\varepsilon^{n-n\beta/p}
=C'\varepsilon^{n(1-\beta/p)}
\longrightarrow0,
\end{align*}
because $n(1-\beta/p)>0$.
This is again a contradiction.
Thus, the exponent $\max\{2,p\}$ cannot be decreased.
\end{proof}

\begin{proof}[\textbf{Proof of Theorem \ref{thm takac conjecture general}}]
Let $d_{\Omega}(x):=\dist(x,\partial\Omega)$.
By the global boundary regularity result of Lieberman
\cite{lieberman1988boundary}, we have
$u_1\in C^{1,\theta}(\overline{\Omega})$
for some $\theta\in(0,1)$.
Moreover, by the strong maximum and boundary point
principles of V\'azquez
\cite[Theorem~5, p.~200]{vazquez1984strong}, we have
\begin{align*}
u_1>0\quad\text{in }\Omega,
\qquad
\frac{\partial u_1}{\partial\nu}<0
\quad\text{on }\partial\Omega,
\end{align*}
where $\nu$ is the outward unit normal.
By continuity, compactness of $\partial\Omega$,
and positivity of $u_1$ in $\Omega$,
there exists $c_0>0$ such that
\begin{align}\label{eq general domain boundary bound}
u_1(x)\geq c_0d_{\Omega}(x)
\quad\text{for all }x\in\Omega.
\end{align}
The global regularity also ensures that
$\|\nabla u_1\|_{L^\infty(\Omega)}<\infty$. Since a bounded connected $C^{2}$ domain is a John domain, we can use the improved Poincar\'e inequality \cite[Theorem~3.1]{drelichman2008improved} for $q=p$ and $\alpha=1$:
\begin{align}\label{eq general domain improved poincare}
\inf_{a\in\mathbb{R}}\norm{v-a}_{L^{p}(\Omega)}
\leq C\norm{d_{\Omega}\nabla v}_{L^{p}(\Omega)}
\end{align}
for every $v\in C_{0}^{1}(\Omega)$. Here and below, $C>0$ denotes a constant depending on $\Omega$ and $p$ that may change from one occurrence to the next. Fix $u\in C_{0}^{\infty}(\Omega)$ and set $v:=u/u_{1}$. Positivity and local $C^{1}$ regularity of $u_{1}$ imply $v\in C_{0}^{1}(\Omega)$ and
\begin{align*}
u_{1}\nabla v=\nabla u-\frac{u}{u_{1}}\nabla u_{1},
\qquad
u^{\parallel}=\frac{\int_{\Omega}u_{1}^{2}v\,dx}{\int_{\Omega}u_{1}^{2}dx}.
\end{align*}
For every $a\in\mathbb{R}$, H\"older's inequality, with $p'=p/(p-1)$, gives
\begin{align*}
|u^{\parallel}-a|
&=\frac{\left|\int_{\Omega}u_{1}^{2}(v-a)dx\right|}{\int_{\Omega}u_{1}^{2}dx}\\
&\leq\frac{\norm{u_{1}^{2}}_{L^{p'}(\Omega)}}{\int_{\Omega}u_{1}^{2}dx}
\norm{v-a}_{L^{p}(\Omega)}.
\end{align*}
The denominator is positive and the numerator is finite because $u_{1}$ is bounded. The triangle inequality therefore yields
\begin{align*}
\norm{v-u^{\parallel}}_{L^{p}(\Omega)}
&\leq\norm{v-a}_{L^{p}(\Omega)}+|\Omega|^{1/p}|a-u^{\parallel}|\\
&\leq\left(1+\frac{|\Omega|^{1/p}\norm{u_{1}^{2}}_{L^{p'}(\Omega)}}{\int_{\Omega}u_{1}^{2}dx}\right)
\norm{v-a}_{L^{p}(\Omega)}.
\end{align*}
Taking the infimum over $a\in\mathbb{R}$ and using (\ref{eq general domain improved poincare}) and (\ref{eq general domain boundary bound}), we obtain
\begin{align*}
\norm{v-u^{\parallel}}_{L^{p}(\Omega)}
\leq C\norm{d_{\Omega}\nabla v}_{L^{p}(\Omega)}
\leq C\norm{u_{1}\nabla v}_{L^{p}(\Omega)}.
\end{align*}
Since $u^{\perp}=u_{1}(v-u^{\parallel})$ and $u^{\parallel}$ is constant, it follows that
\begin{align}
\norm{\nabla u^{\perp}}_{L^{p}(\Omega)}
&\leq\norm{u_{1}\nabla v}_{L^{p}(\Omega)}
+\norm{\nabla u_{1}}_{L^{\infty}(\Omega)}\norm{v-u^{\parallel}}_{L^{p}(\Omega)}\notag\\
&\leq B\left\|u_{1}\nabla\left(\frac{u}{u_{1}}\right)\right\|_{L^{p}(\Omega)}
\label{eq general domain gradient comparison}
\end{align}
for some $B>0$ depending on $\Omega$ and $p$. This replaces the gradient conversion based on log-concavity in the convex-domain proof.

The derivation of (\ref{interm grad 1p2}) uses only the remainder estimate in Theorem \ref{thm poinc 1p2}, the first eigenfunction equation and H\"older's inequality, so it also holds on the present domain. Applying the gradient comparison (\ref{eq general domain gradient comparison}) in (\ref{interm grad 1p2}) yields
\begin{align*}
\int_{\Omega}|\nabla u|^{p}dx-\lambda_{1}(\Omega,p)\int_{\Omega}|u|^{p}dx\geq\frac{\underline{c}_{\mathrm{sing}}(p)}{B^{2}}
\left(\frac{p-1}{2p-1}\right)^{2-p}
\norm{\nabla u}_{L^{p}(\Omega)}^{-(2-p)}
\norm{\nabla u^{\perp}}_{L^{p}(\Omega)}^{2}.
\end{align*}
The projection $u\mapsto u^{\perp}$ is continuous on $W_{0}^{1,p}(\Omega)$ by H\"older's inequality, the boundedness of $u_{1}$ and the ordinary Poincar\'e inequality. Thus, density extends this estimate to every nonzero $u\in W_{0}^{1,p}(\Omega)$.
\end{proof}

\begin{proof}[\textbf{Proof of Theorem \ref{thm upper gradient stability}}] The calculation leading to (\ref{bound grad u1 ratio}) uses only the weak eigenfunction equation (\ref{eq weak eigenfunction}), so it also applies on the present domain. By density, using the positivity and local $C^1$ regularity of $u_1$ to identify the limit, this estimate extends to every $u\in W_0^{1,p}(\Omega)$:
\begin{align}\label{eq upper eigenfunction testing bound}
\norm{\frac{u}{u_1}\nabla u_1}_{L^p(\Omega)}
\leq\frac{p}{p-1}\norm{\nabla u}_{L^p(\Omega)}.
\end{align}
Moreover, the quotient rule gives, almost everywhere in $\Omega$,
\begin{align*}
\left|u_1\nabla\left(\frac{u}{u_1}\right)\right|
=\left|\nabla u-\frac{u}{u_1}\nabla u_1\right|
\leq|\nabla u|+\left|\frac{u}{u_1}\nabla u_1\right|.
\end{align*}
Since $p-2<0$, Theorem \ref{thm poinc 1p2 upper} therefore gives
\begin{align}\label{eq upper deficit quotient bound}
\int_{\Omega}|\nabla u|^{p}dx-\lambda_1(\Omega,p)\int_{\Omega}|u|^{p}dx\leq\overline{c}_{\mathrm{sing}}(p)
\int_{\Omega}\left|u_1\nabla\left(\frac{u}{u_1}\right)\right|^{p}dx.
\end{align}
Fix any $c\in\mathbb{R}$ and put $v=u-cu_1$. Then $v\in W_{0}^{1,p}(\Omega)$ and $\nabla(v/u_1)=\nabla(u/u_1)$. By the triangle inequality and (\ref{eq upper eigenfunction testing bound}), we have
\begin{align*}
\norm{u_1\nabla\left(\frac{u}{u_1}\right)}_{L^{p}(\Omega)}
&=\norm{\nabla v-\frac{v}{u_1}\nabla u_1}_{L^{p}(\Omega)}\\
&\leq\norm{\nabla v}_{L^{p}(\Omega)}
+\norm{\frac{v}{u_1}\nabla u_1}_{L^{p}(\Omega)}\\
&\leq\frac{2p-1}{p-1}\norm{\nabla(u-cu_1)}_{L^{p}(\Omega)}.
\end{align*}
Substituting this estimate into (\ref{eq upper deficit quotient bound}), taking the infimum over $c\in\mathbb{R}$ and dividing by $\norm{\nabla u}_{L^{p}(\Omega)}^{p}>0$ for nonzero $u$ yields (\ref{eq normalized upper gradient stability}). Next, we show that the exponent is optimal.

As in the proof of Theorem \ref{thm gradient stability}, write
\begin{align*}
\delta(u)&:=1-\lambda_1(\Omega,p)
\frac{\norm{u}_{L^{p}(\Omega)}^{p}}
{\norm{\nabla u}_{L^{p}(\Omega)}^{p}},\\
d(u)&:=\inf_{c\in\mathbb{R}}
\norm{\nabla(u-cu_1)}_{L^{p}(\Omega)}.
\end{align*}
To prove optimality, suppose that, for some $\beta>p$,
there exists a constant $C_*>0$, independent of $u$, such that
\begin{align}\label{eq upper stability larger exponent}
\delta(u)\leq C_*
\left(\frac{d(u)}{\norm{\nabla u}_{L^{p}(\Omega)}}\right)^\beta
\end{align}
for every nonzero $u\in W_0^{1,p}(\Omega)$ with $d(u)>0$.
Choose $B=B_r(x_0)$ with $\overline B\subset\Omega$ sufficiently
small that (\ref{eq optimality exterior gradient bound}) holds.
Let $0\neq \varphi\in C_0^\infty(B_1(0))$ and
\begin{align*}
h_\varepsilon(x)
:=\varepsilon\varphi\left(\frac{x-x_0}{\varepsilon}\right),
\qquad
u_\varepsilon:=u_1+h_\varepsilon,
\qquad 0<\varepsilon<r.
\end{align*}
The calculations (\ref{eq optimality scaled gradient norm}),
(\ref{eq optimality shrinking distance}) and
(\ref{eq optimality shrinking gradient bounds})
remain valid for $1<p<2$. In particular, $d(u_\varepsilon)>0$ and,
for sufficiently small $\varepsilon$,
\begin{align*}
\frac{d(u_\varepsilon)}
{\norm{\nabla u_\varepsilon}_{L^p(\Omega)}}
\leq
\frac{2\norm{\nabla\varphi}_{L^p(B_1(0))}}
{\norm{\nabla u_1}_{L^p(\Omega)}}
\varepsilon^{n/p}.
\end{align*}
Using (\ref{eq optimality deficit identity}), we get
\begin{align*}
\norm{\nabla u_\varepsilon}_{L^p(\Omega)}^p
\delta(u_\varepsilon)
&=\int_B
C_p(\nabla u_1+\nabla h_\varepsilon,\nabla h_\varepsilon)\,dx\\
&\quad-\lambda_1(\Omega,p)
\int_B C_p(u_1+h_\varepsilon,h_\varepsilon)\,dx.
\end{align*}
We next estimate the deficit from below using the bounds for
$C_p$ in \cite[Lemmas~2.2 and~2.3]{CT24}. By the triangle inequality and
(\ref{eq optimality scaled gradient}), we have on $B$
\begin{align*}
|\nabla u_1+\nabla h_\varepsilon|+|\nabla u_1|
&\leq 2|\nabla u_1|+|\nabla h_\varepsilon|\\
&\leq 2\norm{\nabla u_1}_{L^\infty(B)}
+\norm{\nabla\varphi}_{L^\infty(B_1(0))}.
\end{align*}
Since $p-2<0$, this implies
\begin{align*}
&\left(|\nabla u_1+\nabla h_\varepsilon|+|\nabla u_1|\right)^{p-2}\\
&\qquad\geq\left(
2\norm{\nabla u_1}_{L^\infty(B)}
+\norm{\nabla\varphi}_{L^\infty(B_1(0))}
\right)^{p-2}.
\end{align*}
The lower estimate in \cite[Lemma~2.2]{CT24} therefore yields
\begin{align}
&C_p(\nabla u_1+\nabla h_\varepsilon,\nabla h_\varepsilon) \nonumber
\\&\quad\geq\underline{c}_{\mathrm{sing}}(p)
\left(|\nabla u_1+\nabla h_\varepsilon|+|\nabla u_1|\right)^{p-2}
|\nabla h_\varepsilon|^2 \nonumber
\\&\quad\geq c|\nabla h_\varepsilon|^2
\qquad\text{on }B, \label{eq upper optimality gradient remainder lower bound}
\end{align}
where
\begin{align*}
c:=\underline{c}_{\mathrm{sing}}(p)
\left(
2\norm{\nabla u_1}_{L^\infty(B)}
+\norm{\nabla\varphi}_{L^\infty(B_1(0))}
\right)^{p-2}>0
\end{align*}
is independent of $\varepsilon$. Similarly, the upper estimate in
\cite[Lemma~2.3]{CT24}, applied to scalar arguments, gives
\begin{align}
C_p(u_1+h_\varepsilon,h_\varepsilon)
&\leq\overline{c}_{\mathrm{sing}}(p)
\left(|u_1+h_\varepsilon|+u_1\right)^{p-2}
|h_\varepsilon|^2 \nonumber
\\&\leq C|h_\varepsilon|^2
\qquad\text{on }B, \label{eq upper optimality scalar remainder upper bound}
\end{align}
with $C>0$ independent of $\varepsilon$ since
$|u_1+h_\varepsilon|+u_1\geq u_1$ and $u_1$ is bounded
away from zero on $\overline B$. Substituting (\ref{eq upper optimality gradient remainder lower bound})
and (\ref{eq upper optimality scalar remainder upper bound}) into
(\ref{eq optimality deficit identity}) with $h=h_\varepsilon$,
and making the same change of variables as in
(\ref{eq optimality scaled gradient norm}), we obtain
\begin{align*}
\norm{\nabla u_\varepsilon}_{L^p(\Omega)}^p
\delta(u_\varepsilon)
&\geq c\int_B|\nabla h_\varepsilon|^2\,dx
-\lambda_1(\Omega,p)C\int_B|h_\varepsilon|^2\,dx\\
&=c\varepsilon^n\int_{B_1(0)}|\nabla\varphi|^2\,dy
-\lambda_1(\Omega,p)C\varepsilon^{n+2}
\int_{B_1(0)}|\varphi|^2\,dy\\
&\geq\frac c2\varepsilon^n
\int_{B_1(0)}|\nabla\varphi|^2\,dy
\end{align*}
for sufficiently small $\varepsilon$. Using (\ref{eq optimality shrinking gradient bounds}) and
adjusting $c>0$ independently of $\varepsilon$, we conclude that
\begin{align*}
\delta(u_\varepsilon)\geq c\varepsilon^n.
\end{align*}
Combining this with
(\ref{eq upper stability larger exponent}) and the normalized
distance estimate gives
\begin{align*}
c\varepsilon^n
\leq\delta(u_\varepsilon)
\leq C_*
\left(
\frac{d(u_\varepsilon)}
{\norm{\nabla u_\varepsilon}_{L^p(\Omega)}}
\right)^\beta
\leq C'\varepsilon^{n\beta/p}.
\end{align*}
Dividing by $\varepsilon^n>0$, we conclude that
\begin{align*}
0<c\leq C'\varepsilon^{n(\beta/p-1)}
\longrightarrow0
\end{align*}
since $\beta>p$, a contradiction. Thus, the exponent $p$
in (\ref{eq normalized upper gradient stability})
cannot be increased.
\end{proof}

\begin{proof}[\textbf{Proof of Theorem \ref{thm fundamental gap}}]
Normalize $u_{1}$ so that $\norm{u_{1}}_{L^{p}(\Omega)}=1$. We first derive a gradient version of the path-intersection argument in \cite[Lemma~6.3]{chen2026fundamental}. Let us define
\begin{align*}
\mathcal{H}_{u_{1}}:=\left\{w\in\mathcal{S}_{p}:\int_{\Omega}|\nabla w|^{p-2}\langle\nabla w,\nabla u_{1}\rangle dx=0\right\}.
\end{align*}
Then, set
\begin{align*}
F(s):=\int_{\Omega}|\nabla\gamma(s)|^{p-2}\langle\nabla\gamma(s),\nabla u_{1}\rangle dx
\end{align*}
for $\gamma\in\Psi(u_{1},-u_{1})$. The map $F$ is continuous since $\gamma$ is continuous in $W_{0}^{1,p}(\Omega)$ and $\xi\mapsto|\xi|^{p-2}\xi$ is continuous from $L^{p}(\Omega)$ to $L^{p/(p-1)}(\Omega)$. Moreover,
\begin{align*}
F(0)=\lambda_{1}(\Omega,p),\qquad F(1)=-\lambda_{1}(\Omega,p).
\end{align*}
Therefore, $F(s_{\gamma})=0$ for some $s_{\gamma}\in(0,1)$, so $w_{\gamma}:=\gamma(s_{\gamma})\in\mathcal{H}_{u_{1}}$. For $w\in\mathcal{H}_{u_{1}}$, the function
\begin{align*}
\Phi(c):=\norm{\nabla(w-cu_{1})}_{L^{p}(\Omega)}^{p}
\end{align*}
is strictly convex and satisfies
\begin{align*}
\Phi'(0)=-p\int_{\Omega}|\nabla w|^{p-2}\langle\nabla w,\nabla u_{1}\rangle dx=0.
\end{align*}
Thus, $c=0$ is its unique minimizer and
\begin{align*}
\inf_{c\in\mathbb{R}}\norm{\nabla(w-cu_{1})}_{L^{p}(\Omega)}^{p}
=\Phi(0)=\norm{\nabla w}_{L^{p}(\Omega)}^{p}.
\end{align*}
Since $w_{\gamma}\in\mathcal{S}_{p}$ and
\begin{align*}
\inf_{c\in\mathbb{R}}\frac{\norm{\nabla(w_{\gamma}-cu_{1})}_{L^{p}(\Omega)}}{\norm{\nabla w_{\gamma}}_{L^{p}(\Omega)}}=1,
\end{align*}
the gradient stability estimate (\ref{eq gradient stability unified}) yields
\begin{align*}
1-\frac{\lambda_{1}(\Omega,p)}{\int_{\Omega}|\nabla w_{\gamma}|^{p}dx}\geq C_{\Omega,p}.
\end{align*}
Since $\lambda_{1}(\Omega,p)>0$, we have $C_{\Omega,p}<1$ and therefore
\begin{align*}
\int_{\Omega}|\nabla w_{\gamma}|^{p}dx\geq\frac{\lambda_{1}(\Omega,p)}{1-C_{\Omega,p}}\geq 0.
\end{align*}
Consequently,
\begin{align*}
\max_{s\in[0,1]}\int_{\Omega}|\nabla\gamma(s)|^{p}dx\geq\frac{\lambda_{1}(\Omega,p)}{1-C_{\Omega,p}}.
\end{align*}
Taking the infimum over $\gamma\in\Psi(u_{1},-u_{1})$ and using (\ref{eigenvalue second}), we obtain
\begin{align}\label{eq eigenvalue ratio bound}
\lambda_{2}(\Omega,p)\geq\frac{\lambda_{1}(\Omega,p)}{1-C_{\Omega,p}},
\end{align}
which proves (\ref{eq unified spectral gap}). Rearranging (\ref{eq eigenvalue ratio bound}) also gives the equivalent form
\begin{align*}
\Gamma(\Omega,p)\geq C_{\Omega,p}\lambda_{2}(\Omega,p).
\end{align*}
Finally, since $0<C_{\Omega,p}<1$, we have
\begin{align*}
\Gamma(\Omega,p)\geq\frac{C_{\Omega,p}}{1-C_{\Omega,p}}\lambda_{1}(\Omega,p)
\geq C_{\Omega,p}\lambda_{1}(\Omega,p).
\end{align*}
Substituting (\ref{eq stability constant lower bound}) into (\ref{eq unified spectral gap}) and using $1/(1-C_{\Omega,p})\geq1$ gives (\ref{eq explicit spectral gap}).
\end{proof}

\begin{proof}[\textbf{Proof of Corollary \ref{cor inradius fundamental gap}}]
By Theorem \ref{thm brasco inradius},
\begin{align*}
\left(\frac{\pi_{p}}{2r_{\Omega}}\right)^{p}
\leq\lambda_{1}(\Omega,p)
\leq\frac{\lambda_{1}(B_{1}(0),p)}{r_{\Omega}^{p}}.
\end{align*}
Using (\ref{eq rho upper bound}), the upper bound in (\ref{eq brasco inradius bounds}) and $2r_{\Omega}\leq\mathrm{diam}(\Omega)$, we obtain
\begin{align}
1+\rho
&<p+1+\lambda_{1}(\Omega,p)^{1/p}\frac{\mathrm{diam}(\Omega)}{\pi_{p}}\notag\\
&\leq p+1+\frac{\lambda_{1}(B_{1}(0),p)^{1/p}\mathrm{diam}(\Omega)}{\pi_{p}r_{\Omega}}\notag\\
&\leq\left(\frac{p+1}{2}+\frac{\lambda_{1}(B_{1}(0),p)^{1/p}}{\pi_{p}}\right)
\frac{\mathrm{diam}(\Omega)}{r_{\Omega}}.
\label{eq geometric rho bound}
\end{align}
Since $\max\{2,p\}>0$, taking reciprocals and raising both sides of (\ref{eq geometric rho bound}) to the power $\max\{2,p\}$ gives
\begin{align}\label{eq geometric reciprocal bound}
\frac{1}{(1+\rho)^{\max\{2,p\}}}
\geq\frac{1}{\left(\frac{p+1}{2}+\frac{\lambda_{1}(B_{1}(0),p)^{1/p}}{\pi_{p}}\right)^{\max\{2,p\}}}
\left(\frac{r_{\Omega}}{\mathrm{diam}(\Omega)}\right)^{\max\{2,p\}}.
\end{align}
The denominator in the definition (\ref{eq stability constant}) is $(1+\rho)^{p}$ when $p\geq2$ and $(1+\rho)^{2}$ when $1<p<2$. Thus, (\ref{eq stability constant}) with (\ref{eq geometric reciprocal bound}) yields
\begin{align}
C_{\Omega,p}
\geq{}\frac{\left(\frac{r_{\Omega}}{\mathrm{diam}(\Omega)}\right)^{\max\{2,p\}}}
{\left(\frac{p+1}{2}+\frac{\lambda_{1}(B_{1}(0),p)^{1/p}}{\pi_{p}}\right)^{\max\{2,p\}}}\times\begin{cases}
c_{\mathrm{deg}}(p),&p\geq2,\\[1ex]
\underline{c}_{\mathrm{sing}}(p)\left(\dfrac{p-1}{2p-1}\right)^{2-p},&1<p<2.
\end{cases}
\label{eq geometric stability constant lower bound}
\end{align}
Next, $0<C_{\Omega,p}<1$ implies $C_{\Omega,p}/(1-C_{\Omega,p})\geq C_{\Omega,p}$. Therefore, (\ref{eq unified spectral gap}) and the lower bound in (\ref{eq brasco inradius bounds}) give
\begin{align}
\Gamma(\Omega,p)
&\geq\frac{C_{\Omega,p}}{1-C_{\Omega,p}}\lambda_{1}(\Omega,p)\notag\\
&\geq C_{\Omega,p}\lambda_{1}(\Omega,p)\notag\\
&\geq C_{\Omega,p}\frac{\pi_{p}^{p}}{2^{p}r_{\Omega}^{p}}.
\label{eq geometric gap eigenvalue bound}
\end{align}
Substituting (\ref{eq geometric stability constant lower bound}) into the last inequality of (\ref{eq geometric gap eigenvalue bound}), we obtain
\begin{align}
\Gamma(\Omega,p)
\geq{}&\frac{\pi_{p}^{p}}{2^{p}\left(\frac{p+1}{2}+\frac{\lambda_{1}(B_{1}(0),p)^{1/p}}{\pi_{p}}\right)^{\max\{2,p\}}}
\frac{r_{\Omega}^{\max\{2,p\}-p}}{\mathrm{diam}(\Omega)^{\max\{2,p\}}}
\notag\\&\times\begin{cases}
c_{\mathrm{deg}}(p),&p\geq2,\\[1ex]
\underline{c}_{\mathrm{sing}}(p)\left(\dfrac{p-1}{2p-1}\right)^{2-p},&1<p<2.
\end{cases}
\label{eq geometric gap expanded}
\end{align}
Here the power of the inradius follows from
\begin{align*}
\frac{1}{r_{\Omega}^{p}}
\left(\frac{r_{\Omega}}{\mathrm{diam}(\Omega)}\right)^{\max\{2,p\}}
=\frac{r_{\Omega}^{\max\{2,p\}-p}}{\mathrm{diam}(\Omega)^{\max\{2,p\}}}.
\end{align*}
By the definition (\ref{eq geometric gap constant}) of $K_{n,p}$, (\ref{eq geometric gap expanded}) becomes
\begin{align*}
\Gamma(\Omega,p)\geq K_{n,p}
\frac{r_{\Omega}^{\max\{2,p\}-p}}{\mathrm{diam}(\Omega)^{\max\{2,p\}}},
\end{align*}
which proves (\ref{eq unified geometric gap}).
\end{proof}

{\bf Acknowledgements.} AI tools were used during the preparation of this manuscript. All mathematical arguments and proofs presented in the paper were developed and checked by the authors.

%{\bf Data availability.} No datasets were generated or analysed for the research carried
%out in this paper.
%{\bf Conflict of interest.} The authors declares that he has no conflict of interest.

% Bibliography embedded for compilation in the document editor.

\bibliographystyle{alpha}
\bibliography{references}

@article{ashbaugh1989optimal,
  title={Optimal lower bound for the gap between the first two eigenvalues of one-dimensional {S}chr{\"o}dinger operators with symmetric single-well potentials},
  author={Ashbaugh, M. S. and Benguria, R.},
  journal={Proc. Amer. Math. Soc.},
  volume={105},
  number={2},
  pages={419--424},
  year={1989},
}

@article{amato2025geometric,
  title={The geometric size of the fundamental gap},
  author={Amato, V. and Bucur, D. and Fragal{\`a}, I.},
  journal={arXiv preprint arXiv:2407.01341},
  year={2024},
}

@article{andrews2009lipschitz,
  title={{L}ipschitz bounds for solutions of quasilinear parabolic equations in one space variable},
  author={Andrews, B. and Clutterbuck, J.},
  journal={J. Differential Equations},
  volume={246},
  number={11},
  pages={4268--4283},
  year={2009},
}

@article{andrews2009time,
  title={Time-interior gradient estimates for quasilinear parabolic equations},
  author={Andrews, B. and Clutterbuck, J.},
  journal={Indiana Univ. Math. J.},
  volume={58},
  number={1},
  pages={351--380},
  year={2009},
}

@article{andrews2011proof,
  title={Proof of the fundamental gap conjecture},
  author={Andrews, B. and Clutterbuck, J.},
  journal={J. Amer. Math. Soc.},
  volume={24},
  number={3},
  pages={899--916},
  year={2011},
}

@article{andrews2021fundamental,
  title={The fundamental gap for a one-dimensional {S}chr{\"o}dinger operator with {R}obin boundary conditions},
  author={Andrews, B. and Clutterbuck, J. and Hauer, D.},
  journal={Proc. Amer. Math. Soc.},
  volume={149},
  number={4},
  pages={1481--1493},
  year={2021},
}

@article{andrews2012eigenvalue,
  title={Eigenvalue comparison on {B}akry--{\'E}mery manifolds},
  author={Andrews, B. and Ni, L.},
  journal={Comm. Partial Differential Equations},
  volume={37},
  number={11},
  pages={2081--2092},
  year={2012},
}

@article{anane1987simplicite,
  title={Simplicit{\'e} et isolation de la premi{\`e}re valeur propre du {$p$}-laplacien avec poids},
  author={Anane, A.},
  journal={C. R. Acad. Sci. Paris S{\'e}r. I Math.},
  volume={305},
  number={16},
  pages={725--728},
  year={1987},
}

@misc{ashbaugh2006sharp,
  author = {Ashbaugh, M. S.},
  title = {The {F}undamental {G}ap},
  year = {2006},
  howpublished = {Notes for a workshop at the
                  American Institute of Mathematics},
  note = {Available at
          \url{https://www.aimath.org/WWN/loweigenvalues/gap.pdf}}
}

@incollection{anane1996second,
  title={On the second eigenvalue of the {$p$}-{L}aplacian},
  author={Anane, A. and Tsouli, N.},
  booktitle={Nonlinear partial differential equations ({F}{\`e}s, 1994)},
  series={Pitman Res. Notes Math. Ser.},
  volume={343},
  pages={1--9},
  year={1996},
  publisher={Longman},
}

@article{apseit2025sharp,
  title={Sharp remainder of the ${L}^{p}$-{P}oincar\'e inequality for {B}aouendi--{G}rushin vector fields},
  author={Apseit, K. and Yessirkegenov, N. and Zhangirbayev, A.},
  journal={arXiv preprint arXiv:2507.01681},
  year={2025},
}

@article{bourni2022vanishing,
  title={The vanishing of the fundamental gap of convex domains in $\mathbb{H}^{N}$},
  author={Bourni, T. and Clutterbuck, J. and Nguyen, X. H. and Stancu, A. and Wei, G. and Wheeler, V.-M.},
  journal={Ann. Henri Poincar{\'e}},
  volume={23},
  number={2},
  pages={595--614},
  year={2022},
}

@article{bonforte2025stability,
  title={Stability in {G}agliardo--{N}irenberg--{S}obolev inequalities: {F}lows, regularity and the entropy method},
  author={Bonforte, M. and Dolbeault, J. and Nazaret, B. and Simonov, N.},
  journal={Mem. Amer. Math. Soc.},
  volume={308},
  number={1554},
  year={2025},
}

@article{brasco2015faber,
  title={{F}aber--{K}rahn inequalities in sharp quantitative form},
  author={Brasco, L. and De Philippis, G. and Velichkov, B.},
  journal={Duke Math. J.},
  volume={164},
  number={9},
  pages={1777--1831},
  year={2015},
}

@article{bianchi1991note,
  title={A note on the {S}obolev inequality},
  author={Bianchi, G. and Egnell, H.},
  journal={J. Funct. Anal.},
  volume={100},
  number={1},
  pages={18--24},
  year={1991},
}

@article{brasco2013hong,
  title={On the {H}ong--{K}rahn--{S}zeg{\Haccent{o}} inequality for the {$p$}-{L}aplace operator},
  author={Brasco, L. and Franzina, G.},
  journal={Manuscripta Math.},
  volume={141},
  number={3--4},
  pages={537--557},
  year={2013},
}

@article{bobkov2023improved,
  title={Improved {F}riedrichs inequality for a subhomogeneous embedding},
  author={Bobkov, V. and Kolonitskii, S.},
  journal={J. Math. Anal. Appl.},
  volume={527},
  number={1},
  pages={127383},
  year={2023},
}

@article{brascamp1976extensions,
  title={On extensions of the {B}runn--{M}inkowski and {P}r{\'e}kopa--{L}eindler theorems, including inequalities for log concave functions, and with an application to the diffusion equation},
  author={Brascamp, H. J. and Lieb, E. H.},
  journal={J. Funct. Anal.},
  volume={22},
  number={4},
  pages={366--389},
  year={1976},
}

@article{brezis1985sobolev,
  title={{S}obolev inequalities with remainder terms},
  author={Brezis, H. and Lieb, E. H.},
  journal={J. Funct. Anal.},
  volume={62},
  number={1},
  pages={73--86},
  year={1985},
}

@article{bobkov2026nonljusternik,
  title={Non-{L}justernik--{S}chnirelman eigenvalues of the pure {$p$}-{L}aplacian exist},
  author={Bobkov, V.},
  journal={arXiv preprint arXiv:2604.01138},
  year={2026},
}

@article{brasco2018principal,
  title={On principal frequencies and inradius in convex sets},
  author={Brasco, L.},
  journal={Bruno Pini Math. Anal. Semin.},
  volume={9},
  number={1},
  pages={78--101},
  year={2018},
}

@incollection{browder1970nonlinear,
  title={Nonlinear eigenvalue problems and group invariance},
  author={Browder, F. E.},
  booktitle={Functional Analysis and Related Fields: {P}roceedings of a Conference in honor of Professor {M}arshall {S}tone, held at the {U}niversity of {C}hicago, {M}ay 1968},
  pages={1--58},
  year={1970},
  publisher={Springer},
}

@article{brezis1997blow,
  title={Blow-up solutions of some nonlinear elliptic problems},
  author={Brezis, H. and V{\'a}zquez, J. L.},
  journal={Rev. Mat. Univ. Complut. Madrid},
  volume={10},
  number={2},
  pages={443--469},
  year={1997},
}

@article{carlen2017duality,
  title={Duality and stability for functional inequalities},
  author={Carlen, E. A.},
  journal={Ann. Fac. Sci. Toulouse Math. (6)},
  volume={26},
  number={2},
  pages={319--350},
  year={2017},
}

@article{chen2020remarks,
  title={Remarks on the one-dimensional sloshing problem involving the $p$-{L}aplacian operator},
  author={Chen, W. C. and Cheng, Y. H.},
  journal={Turkish J. Math.},
  volume={44},
  number={4},
  pages={1376--1387},
  year={2020},
}

@article{carlen2013stability,
  title={Stability for a {GNS} inequality and the log-{HLS} inequality, with application to the critical mass {K}eller--{S}egel equation},
  author={Carlen, E. A. and Figalli, A.},
  journal={Duke Math. J.},
  volume={162},
  number={3},
  pages={579--625},
  year={2013},
}

@article{cazacu2024caffarelli,
  title={{C}affarelli--{K}ohn--{N}irenberg identities, inequalities and their stabilities},
  author={Cazacu, C. and Flynn, J. and Lam, N. and Lu, G.},
  journal={J. Math. Pures Appl.},
  volume={182},
  pages={253--284},
  year={2024},
}

@article{chen2026fundamental,
  title={Fundamental {G}aps for the {D}irichlet $p$-{L}aplacian on {C}onvex {D}omains},
  author={Chen, R. and Hauer, D.},
  journal={arXiv preprint arXiv:2608.13443},
  year={2026},
}

@article{chen2026two,
  title={Two-point modulus for the logarithmic $p$-flux of the first {D}irichlet eigenfunction of the $p$-{L}aplacian},
  author={Chen, R.},
  journal={arXiv preprint arXiv:2608.29195},
  year={2026},
}

@article{clutterbuck2025constant,
  title={Constant potentials do not minimise the fundamental gap on convex domains in hyperbolic space},
  author={Clutterbuck, J. and J{\"a}ckel, F. and Nguyen, X. H.},
  journal={arXiv preprint arXiv:2512.17103},
  year={2025},
}

@article{cazacu2024hardy,
  title={{H}ardy inequalities for magnetic $p$-{L}aplacians},
  author={Cazacu, C. and Krej{\v{c}}i{\v{r}}{\'\i}k, D. and Lam, N. and Laptev, A.},
  journal={Nonlinearity},
  volume={37},
  number={3},
  pages={035004},
  year={2024},
}

@article{chen2024stability,
  title={Stability of {H}ardy--{L}ittlewood--{S}obolev inequalities with explicit lower bounds},
  author={Chen, L. and Lu, G. and Tang, H.},
  journal={Adv. Math.},
  volume={450},
  pages={109778},
  year={2024},
}

@article{chen2025optimal,
  title={Optimal asymptotic lower bound for stability of fractional {S}obolev inequality and the global stability of log-{S}obolev inequality on the sphere},
  author={Chen, L. and Lu, G. and Tang, H.},
  journal={Adv. Math.},
  volume={479},
  pages={110438},
  year={2025},
}

@article{chen2026optimal,
  title={Optimal stability of {H}ardy--{L}ittlewood--{S}obolev and {S}obolev inequalities of arbitrary orders with dimension-dependent constants},
  author={Chen, L. and Lu, G. and Tang, H.},
  journal={Math. Ann.},
  volume={394},
  year={2026},
}

@article{chen2025asymptotically,
  title={Asymptotically sharp stability of {S}obolev inequalities on the {H}eisenberg group with dimension-dependent constants},
  author={Chen, L. and Lu, G. and Tang, H. and Wang, B.},
  journal={J. Math. Pures Appl.},
  volume={206},
  pages={103832},
  year={2025},
}

@article{cheng2014dual,
  title={The dual eigenvalue problems for $p$-{L}aplacian},
  author={Cheng, Y. H. and Lian, W. C. and Wang, W. C.},
  journal={Acta Math. Hungar.},
  volume={142},
  number={1},
  pages={132--151},
  year={2014},
}

@article{cicalese2013droplet,
  title={Droplet minimizers of an isoperimetric problem with long-range interactions},
  author={Cicalese, M. and Spadaro, E.},
  journal={Comm. Pure Appl. Math.},
  volume={66},
  number={8},
  pages={1298--1333},
  year={2013},
}

@article{CT24,
  title={Remainder terms of ${L}^{p}$-{H}ardy inequalities with magnetic fields: {T}he case $1<p<2$},
  author={Chen, X. P. and Tang, C. L.},
  journal={J. Geom. Anal.},
  volume={35},
  number={12},
  pages={384},
  year={2025},
}

@article{cuesta1999beginning,
  title={The beginning of the {F}u{\v{c}}{\'\i}k spectrum for the {$p$}-{L}aplacian},
  author={Cuesta, M. and de Figueiredo, D. and Gossez, J. P.},
  journal={J. Differential Equations},
  volume={159},
  number={1},
  pages={212--238},
  year={1999},
}

@inproceedings{cuesta2000fucik,
  author = {Cuesta, M.},
  title = {On the {Fu\v{c}{\'\i}k} Spectrum of
           the Laplacian and the {$p$}-Laplacian},
  booktitle = {Proceedings of the Seminar in
               Differential Equations, Kvilda},
  editor = {Dr{\'a}bek, Pavel},
  pages = {67--96},
  publisher = {University of West Bohemia},
  address = {Pilsen},
  year = {2000}
}

@article{drelichman2008improved,
  title={Improved {P}oincar{\'e} inequalities with weights},
  author={Drelichman, I. and Dur{\'a}n, R. G.},
  journal={J. Math. Anal. Appl.},
  volume={347},
  number={1},
  pages={286--293},
  year={2008},
}

@article{dolbeault2025sharp,
  title={Sharp stability for {S}obolev and log-{S}obolev inequalities, with optimal dimensional dependence},
  author={Dolbeault, J. and Esteban, M. J. and Figalli, A. and Frank, R. L. and Loss, M.},
  journal={Camb. J. Math.},
  volume={13},
  number={2},
  pages={359--430},
  year={2025},
}

@article{do2023lp,
  title={{$L^p$}-{C}affarelli--{K}ohn--{N}irenberg inequalities and their stabilities},
  author={Do, A. X. and Flynn, J. and Lam, N. and Lu, G.},
  journal={arXiv preprint arXiv:2310.07083},
  year={2023},
}

@article{dai2025sharp,
  title={Sharp gradient stability for a class of {H}ardy--{S}obolev--{M}az'ya inequalities},
  author={Dai, W. and Fu, J. and Zhang, A.},
  journal={arXiv preprint arXiv:2509.00814},
  year={2025},
}

@article{do2024scale,
  title={Scale-dependent {P}oincar{\'e} inequalities, log-{S}obolev inequality and the stability of the {H}eisenberg uncertainty principle on the hyperbolic space},
  author={Do, A. X. and Ganguly, D. and Lam, N. and Lu, G.},
  journal={arXiv preprint arXiv:2410.21039},
  year={2024},
}

@article{do2026sharp,
  title={Sharp stability of the {H}eisenberg uncertainty principle: {S}econd-order and curl-free field cases},
  author={Do, A. X. and Lam, N. and Lu, G.},
  journal={J. Funct. Anal.},
  volume={290},
  number={7},
  pages={111321},
  year={2026},
}

@article{drabek1999resonance,
  title={Resonance problems for the $p$-{L}aplacian},
  author={Dr{\'a}bek, P. and Robinson, S. B.},
  journal={J. Funct. Anal.},
  volume={169},
  number={1},
  pages={189--200},
  year={1999},
}

@article{drabek2007plaplacian,
  title={The {$p$}-{L}aplacian, mascot of nonlinear analysis},
  author={Dr{\'a}bek, P.},
  journal={Acta Math. Univ. Comenian. (N.S.)},
  volume={76},
  number={1},
  pages={85--98},
  year={2007},
}

@article{dai2021fundamental,
  title={Fundamental gap estimate for convex domains on sphere---the case $n=2$},
  author={Dai, X. and Seto, S. and Wei, G.},
  journal={Comm. Anal. Geom.},
  volume={29},
  number={5},
  pages={1095--1125},
  year={2021},
}

@article{drabek2007poincare,
  title={{P}oincar{\'e} inequality and {P}alais--{S}male condition for the {$p$}-{L}aplacian},
  author={Dr{\'a}bek, P. and Tak{\'a}{\v{c}}, P.},
  journal={Calc. Var. Partial Differential Equations},
  volume={29},
  number={1},
  pages={31--58},
  year={2007},
}

@article{dolbeault2013improved,
  title={Improved interpolation inequalities, relative entropy and fast diffusion equations},
  author={Dolbeault, J. and Toscani, G.},
  journal={Ann. Inst. H. Poincar{\'e} Anal. Non Lin{\'e}aire},
  volume={30},
  number={5},
  pages={917--934},
  year={2013},
}

@article{deng2024gradient,
  title={Gradient stability of {C}affarelli--{K}ohn--{N}irenberg inequality involving weighted $p$-{L}aplace},
  author={Deng, S. and Tian, X.},
  journal={arXiv preprint arXiv:2401.04129},
  year={2024},
}

@article{deng2025caffarelli,
  title={{C}affarelli--{K}ohn--{N}irenberg-type inequalities related to weighted $p$-{L}aplace equations},
  author={Deng, S. and Tian, X.},
  journal={J. Funct. Anal.},
  volume={288},
  number={9},
  pages={110867},
  year={2025},
}

@article{deng2025stability,
  title={On the stability of the {H}ardy--{S}obolev inequality involving $p$-{L}aplace},
  author={Deng, S. and Tian, X.},
  journal={Nonlinearity},
  volume={38},
  number={8},
  pages={085016},
  year={2025},
}

@book{evans2022partial,
  title={Partial differential equations},
  author={Evans, L. C.},
  series={Grad. Stud. Math.},
  volume={19},
  edition={2},
  year={2010},
  publisher={Amer. Math. Soc.},
  address={Providence, RI},
}

@article{fan2026sharp,
  title={Sharp quantitative stability for the affine $p$-{S}obolev inequality, {P}art {I}: {T}he case $2\leq p<n$},
  author={Fan, S. and Li, G.-D. and Zhang, J.},
  journal={arXiv preprint arXiv:2606.09555},
  year={2026},
}

@article{fan2026sharp2,
  title={Sharp quantitative stability for the affine $p$-{S}obolev inequality, {P}art {II}: {T}he case $1<p<2$},
  author={Fan, S. and Li, G.-D. and Zhang, J.},
  journal={arXiv preprint arXiv:2609.28215},
  year={2026},
}

@article{figalli2011shape,
  title={On the shape of liquid drops and crystals in the small mass regime},
  author={Figalli, A. and Maggi, F.},
  journal={Arch. Ration. Mech. Anal.},
  volume={201},
  number={1},
  pages={143--207},
  year={2011},
}

@article{figalli2013isoperimetric,
  title={On the isoperimetric problem for radial log-convex densities},
  author={Figalli, A. and Maggi, F.},
  journal={Calc. Var. Partial Differential Equations},
  volume={48},
  number={3},
  pages={447--489},
  year={2013},
}

@article{fusco2008sharp,
  title={The sharp quantitative isoperimetric inequality},
  author={Fusco, N. and Maggi, F. and Pratelli, A.},
  journal={Ann. of Math.},
  volume={168},
  number={3},
  pages={941--980},
  year={2008},
}

@book{fucik2006spectral,
  title={Spectral analysis of nonlinear operators},
  author={Fu{\v{c}}{\'\i}k, S. and Ne{\v{c}}as, J. and Sou{\v{c}}ek, J. and Sou{\v{c}}ek, V.},
  series={Lecture Notes in Math.},
  volume={346},
  year={1973},
  publisher={Springer-Verlag},
  address={Berlin--New York},
}

@article{ferone2012remark,
  title={A remark on optimal weighted {P}oincar{\'e} inequalities for convex domains},
  author={Ferone, V. and Nitsch, C. and Trombetti, C.},
  journal={Atti Accad. Naz. Lincei Rend. Lincei Mat. Appl.},
  volume={23},
  number={4},
  pages={467--475},
  year={2012},
}

@article{frank2024sharp,
  title={The sharp $\sigma_{2}$-curvature inequality on the sphere in quantitative form},
  author={Frank, R. L. and Peteranderl, J. W.},
  journal={arXiv preprint arXiv:2412.12819},
  year={2024},
}

@article{frank2025sharp,
  author  = {Frank, R. L. and Peteranderl, J. and Read, L.},
  title   = {Sharp Quantitative Integral Inequalities for Harmonic Extensions},
  journal = {Ann. Sc. Norm. Super. Pisa Cl. Sci. (5)},
  year    = {2026},
}

@article{fleckinger2002improved,
  title={An improved {P}oincar{\'e} inequality and the $p$-{L}aplacian at resonance for $p > 2$},
  author={Fleckinger-Pell{\'e}, J. and Tak{\'a}{\v{c}}, P.},
  journal={Adv. Differential Equations},
  volume={7},
  number={8},
  pages={951--971},
  year={2002},
}

@article{figalli2023sharp,
  title={Sharp quantitative stability of the {B}runn--{M}inkowski inequality},
  author={Figalli, A. and van Hintum, P. and Tiba, M.},
  journal={arXiv preprint arXiv:2310.20643},
  year={2023},
}

@article{figalli2022sharp,
  title={Sharp gradient stability for the {S}obolev inequality},
  author={Figalli, A. and Zhang, Y. R. Y.},
  journal={Duke Math. J.},
  volume={171},
  number={12},
  pages={2407--2459},
  year={2022},
}

@article{figalli2022strong,
  title={Strong stability for the {W}ulff inequality with a crystalline norm},
  author={Figalli, A. and Zhang, Y. R. Y.},
  journal={Comm. Pure Appl. Math.},
  volume={75},
  number={2},
  pages={422--446},
  year={2022},
}

@article{gelfand1963some,
  title={Some problems in the theory of quasi-linear equations},
  author={Gelfand, I. M.},
  journal={Amer. Math. Soc. Transl. (2)},
  volume={29},
  pages={295--381},
  year={1963},
}

@book{GilbargTrudinger1977,
  title={Elliptic Partial Differential Equations of Second Order},
  author={Gilbarg, D. and Trudinger, N. S.},
  year={1977},
  publisher={Springer-Verlag},
  address={New York--Berlin--Heidelberg},
}

@article{huang2025lp,
  title={On {$L^p$}-{H}ardy inequalities for magnetic $p$-{L}aplacians},
  author={Huang, Y. C. and Tong, X.},
  journal={arXiv preprint arXiv:2508.09483},
  year={2025},
}

@article{he2020fundamental,
  title={Fundamental gap of convex domains in the spheres},
  author={He, C. and Wei, G. and Zhang, Q. S.},
  journal={Amer. J. Math.},
  volume={142},
  number={4},
  pages={1161--1191},
  year={2020},
}

@article{lavine1994eigenvalue,
  title={The eigenvalue gap for one-dimensional convex potentials},
  author={Lavine, R.},
  journal={Proc. Amer. Math. Soc.},
  volume={121},
  number={3},
  pages={815--821},
  year={1994},
}

@article{lieberman1988boundary,
  title={Boundary regularity for solutions of degenerate elliptic equations},
  author={Lieberman, G. M.},
  journal={Nonlinear Anal.},
  volume={12},
  number={11},
  pages={1203--1219},
  year={1988},
}

@article{lindqvist1990equation,
  title={On the equation $\mathrm{div} \left( |\nabla u|^{p-2} \nabla u \right) + \lambda |u|^{p-2} u = 0$},
  author={Lindqvist, P.},
  journal={Proc. Amer. Math. Soc.},
  volume={109},
  number={1},
  pages={157--164},
  year={1990},
}

@article{ignat2025optimal,
  title={Optimal convergence rates for the finite element approximation of the {S}obolev constant},
  author={Ignat, L. I. and Zuazua, E.},
  journal={Found. Comput. Math. (to appear)},
  year={2026},
}

@article{ni2013estimates,
  title={Estimates on the modulus of expansion for vector fields solving nonlinear equations},
  author={Ni, L.},
  journal={J. Math. Pures Appl.},
  volume={99},
  number={1},
  pages={1--16},
  year={2013},
}

@article{oden1999spectral,
  title={Spectral gap estimates on compact manifolds},
  author={Oden, K. and Sung, C. J. and Wang, J.},
  journal={Trans. Amer. Math. Soc.},
  volume={351},
  number={9},
  pages={3533--3548},
  year={1999},
}

@article{peteranderl2026sharp,
  title={The sharp $\sigma_k$-curvature inequality on locally conformally flat manifolds in quantitative form},
  author={Peteranderl, J. W.},
  journal={arXiv preprint arXiv:2609.13133},
  year={2026},
}

@article{sakaguchi1987concavity,
  title={Concavity properties of solutions to some degenerate quasilinear elliptic {D}irichlet problems},
  author={Sakaguchi, S.},
  journal={Ann. Scuola Norm. Sup. Pisa Cl. Sci. (4)},
  volume={14},
  number={3},
  pages={403--421},
  year={1987},
}

@book{struwe2008variational,
  title={Variational methods. Applications to nonlinear partial differential equations and {H}amiltonian systems},
  author={Struwe, M.},
  series={Ergebnisse der Mathematik und ihrer Grenzgebiete. 3. Folge [Results in Mathematics and Related Areas. 3rd Series], A Series of Modern Surveys in Mathematics},
  volume={34},
  edition={4},
  year={2008},
  publisher={Springer-Verlag},
  address={Berlin},
}

@article{seto2019sharp,
  title={Sharp fundamental gap estimate on convex domains of sphere},
  author={Seto, S. and Wang, L. and Wei, G.},
  journal={J. Differential Geom.},
  volume={112},
  number={2},
  pages={347--389},
  year={2019},
}

@article{singer1985estimate,
  title={An estimate of the gap of the first two eigenvalues in the {S}chr{\"o}dinger operator},
  author={Singer, I. M. and Wong, B. and Yau, S. T. and Yau, S. S. T.},
  journal={Ann. Scuola Norm. Sup. Pisa Cl. Sci. (4)},
  volume={12},
  number={2},
  pages={319--333},
  year={1985},
}

@article{shaimerdenov2025sharp,
  title={Sharp remainder terms and stability of weighted {H}ardy--{P}oincar{\'e} and {H}eisenberg--{P}auli--{W}eyl inequalities related to the {B}aouendi--{G}rushin operator},
  author={Shaimerdenov, Y. and Yessirkegenov, N. and Zhangirbayev, A.},
  journal={arXiv preprint arXiv:2508.16380},
  year={2025},
}

@article{szulkin1988ljusternik,
  title={{L}justernik--{S}chnirelmann theory on ${C}^1$-manifolds},
  author={Szulkin, A.},
  journal={Ann. Inst. H. Poincar{\'e} Anal. Non Lin{\'e}aire},
  volume={5},
  number={2},
  pages={119--139},
  year={1988},
}

@article{vazquez1984strong,
  title={A strong maximum principle for some quasilinear elliptic equations},
  author={V{\'a}zquez, J. L.},
  journal={Appl. Math. Optim.},
  volume={12},
  pages={191--202},
  year={1984},
}

@article{vandenberg1983condensation,
  title={On condensation in the free-boson gas and the spectrum of the {L}aplacian},
  author={van den Berg, M.},
  journal={J. Stat. Phys.},
  volume={31},
  number={3},
  pages={623--637},
  year={1983},
}

@article{wang2026stability,
  title={Stability of spinorial {S}obolev inequalities on $\mathbb{S}^{n}$},
  author={Wang, G. and Zhang, M.},
  journal={Adv. Math.},
  volume={505},
  pages={111318},
  year={2026},
}

@article{wen2026maximal,
  title={Maximal gaps of eigenvalues of one-dimensional $p$-{L}aplacian},
  author={Wen, Z. and Zhou, L.},
  journal={Mediterr. J. Math.},
  volume={23},
  number={1},
  pages={47},
  year={2026},
}

@article{yau1987nonlinear,
  title={Nonlinear analysis in geometry},
  author={Yau, S. T.},
  journal={Enseign. Math. (2)},
  volume={33},
  number={1--2},
  pages={109--158},
  year={1987},
}

@article{yu1986lower,
  title={Lower bounds of the gap between the first and second eigenvalues of the {S}chr{\"o}dinger operator},
  author={Yu, Q. H. and Zhong, J. Q.},
  journal={Trans. Amer. Math. Soc.},
  volume={294},
  number={1},
  pages={341--349},
  year={1986},
}

@article{yessirkegenov2026stability,
  title={Stability of the {$L^p$}-{P}oincar{\'e} inequality for the {L}ebesgue measure and {G}aussian probability measure with explicit geometric dependence and applications to spectral gaps},
  author={Yessirkegenov, N. and Zhangirbayev, A.},
  journal={arXiv preprint arXiv:2602.05968},
  year={2026},
}

@article{zhou2026stability,
  title={The stability on the {C}affarelli--{K}ohn--{N}irenberg and {H}ardy-type inequalities and beyond},
  author={Zhou, Y. and Zou, W.},
  journal={J. Differential Equations},
  volume={450},
  pages={113738},
  year={2026},
}

@article{BrascoPrinariZagati2022,
  author = {Brasco, L. and Prinari, F. and Zagati, A. C.},
  title = {A comparison principle for the {L}ane-{E}mden equation and applications to geometric estimates},
  journal = {Nonlinear Anal.},
  volume = {220},
  year = {2022},
  pages = {112847},
  doi = {10.1016/j.na.2022.112847},
}

@article{bonforte2023constructive,
  title = {Constructive stability results in interpolation inequalities and explicit improvements of decay rates of fast diffusion equations},
  author = {Bonforte, M. and Dolbeault, J. and Nazaret, B. and Simonov, N.},
  journal = {Discrete Contin. Dyn. Syst.},
  volume = {43},
  number = {3\&4},
  pages = {1070--1089},
  year = {2023},
}

@article{dolbeault2021functional,
  title = {Functional inequalities: {N}onlinear flows and entropy methods as a tool for obtaining sharp and constructive results},
  author = {Dolbeault, J.},
  journal = {Milan J. Math.},
  volume = {89},
  pages = {355--386},
  year = {2021},
  publisher = {Springer},
}

@article{brigati2024stability,
  title = {Stability for the logarithmic {S}obolev inequality},
  author = {Brigati, G. and Dolbeault, J. and Simonov, N.},
  journal = {J. Funct. Anal.},
  volume = {287},
  number = {8},
  pages = {110562},
  year = {2024},
  publisher = {Elsevier},
}

@article{brigati2024logarithmic,
  title = {Logarithmic {S}obolev and interpolation inequalities on the sphere: {C}onstructive stability results},
  author = {Brigati, G. and Dolbeault, J. and Simonov, N.},
  journal = {Ann. Inst. H. Poincar{\'e} Anal. Non Lin{\'e}aire},
  volume = {41},
  pages = {1289--1321},
  year = {2024},
  publisher = {EMS Press},
}

@article{Koenig2025,
  author = {K{\"o}nig, T.},
  title = {Stability for the {S}obolev inequality: {E}xistence of a minimizer},
  journal = {J. Eur. Math. Soc. (JEMS)},
  year = {2025},
}

@article{Koenig2023,
  author = {K{\"o}nig, T.},
  title = {On the sharp constant in the {B}ianchi--{E}gnell stability inequality},
  journal = {Bull. Lond. Math. Soc.},
  volume = {55},
  year = {2023},
  pages = {2070--2075},
}

\end{document}